\documentclass[a4paper,10pt]{article}

\usepackage{epsfig}
\usepackage{amsmath}
\usepackage{amssymb}
\usepackage{amsthm}
\usepackage{mathrsfs}
\usepackage{color}
\usepackage{amscd}

\usepackage{mathtools} 					
\usepackage{hyperref}						
 \usepackage{booktabs}					

\usepackage{amsthm}
\newtheorem{theorem}{Theorem}[section]

\newtheorem{lemma}[theorem]{Lemma}
\newtheorem{proposition}[theorem]{Proposition}
\newtheorem{corollary}[theorem]{Corollary}
\newtheorem{remark}[theorem]{Remark}
\newtheorem{example}[theorem]{Example}

\theoremstyle{remark}{
}
\theoremstyle{definition}{
\newtheorem{definition}[theorem]{Definition}}

\usepackage{amsfonts}

\newcommand{\hh}{{\mathbb{H}}}

\newcommand{\rr}{{\mathbb{R}}}
\newcommand{\zz}{{\mathbb{Z}}}
\newcommand{\nn}{{\mathbb{N}}}

\newcommand{\ext}{\textnormal{ext}}

\newcommand{\s}{{\mathbb{S}}}

\newcommand\re{\operatorname{Re}}
\newcommand\im{\operatorname{Im}}

\newcommand{\hslashslash}{%
  \raisebox{.9ex}{%
    \scalebox{.7}{%
      \rotatebox[origin=c]{18}{$-$}%
    }%
  }%
}
\newcommand{\fslash}{%
  {%
   \vphantom{f}%
   \ooalign{\kern.05em\smash{\hslashslash}\hidewidth\cr$f$\cr}%
   \kern.05em
  }%
}

\title{\bf Geometric function theory over quaternionic slice domains}

\author{Graziano Gentili, Caterina Stoppato\\
\\
\small Dipartimento di Matematica e Informatica ``U. Dini'', Universit\`a degli Studi di Firenze \\
\small Viale Morgagni 67/A, I-50134 Firenze, Italy\\
\small graziano.gentili@unifi.it, caterina.stoppato@unifi.it}

\date{  }

\begin{document}

\maketitle


\begin{abstract}
The theory of quaternionic slice regular functions was introduced in 2006 and successfully developed for about a decade over symmetric slice domains, which appeared to be the natural setting for their study. Some recent articles paved the way for a further development of the theory: namely, the study of slice regular functions on slice domains that are not necessarily symmetric. The present work is a panorama of geometric function theory in this new context, where new phenomena appear. For instance, the nature of the zero sets can be drastically different than in the symmetric case. The work includes differential, algebraic, topological properties, as well as integral and series representations, of slice regular functions over slice domains.
\end{abstract}


\thanks{\small \noindent{\bf Acknowledgements.} This work was partly supported by INdAM, through: GNSAGA; INdAM project ``Hypercomplex function theory and applications''. It was also partly supported by MIUR, through the projects: Finanziamento Premiale FOE 2014 ``Splines for accUrate NumeRics: adaptIve models for Simulation Environments''; PRIN 2017 ``Real and complex manifolds: topology, geometry and holomorphic dynamics''.}


\section{Introduction}\label{sec:introduction}

The works~\cite{cras,advances} presented a possible quaternionic analog of the theory of holomorphic complex functions, namely, the theory of quaternionic \emph{slice regular functions}. The theory, after a first phase focused on Euclidean balls centered at the origin, made a step forward after the work~\cite{advancesrevised}. This work presented several foundational results, which made the study of slice regular functions interesting on a larger class of domains: the so-called \emph{symmetric slice domains}. On such domains, the theory of slice regular functions turned out to be rich and to have several useful applications to open problems in other areas of mathematics. For a panorama, see~\cite{librospringer}, as well as~\cite{librodaniele2} and subsequent articles.

The recent work~\cite{douren1} showed that slice regular functions on slice domains that are not symmetric do not necessarily extend to the so-called symmetric completions of their domains. This motivated the subsequent works~\cite{douren2,dourensabadini,localrepresentation}. A full development of a theory of slice regular functions on slice domains that are not necessarily symmetric became an intriguing problem. In the present work, we tackle this problem and prove local versions of many properties valid over symmetric slice domains. Besides expected results, we find new, unexpected phenomena. For instance, the zero sets of slice regular functions can be drastically different than in the case of symmetric slice domains: in that case, the zero sets consist of isolated zeros or isolated $2$-spheres of a special type; on general slice domains, each such $2$-sphere may include areas where the function vanishes identically and areas where it does not. While every zero of a slice regular function $g$ can still be factored out, not every factor implies a zero of $g$. These ``ghost'' zeros can, however, play a role in the zero sets of regular products $f*g$ and of the regular conjugate $g^c$. These new phenomena concerning zeros are reflected in the classification of singularities of slice regular functions on slice domains. Both as a tool for the present study and for its independent interest, we prove in this work a stronger version of the Local Extension Theorem, namely~\cite[Theorem 3.2]{localrepresentation}.

The paper is structured as follows. Section~\ref{sec:preliminaries} is devoted to preliminary material. In Section~\ref{sec:differential} we begin the study of real differentials of slice regular functions over a slice domain and in Section~\ref{sec:algebraic} we construct an algebraic structure on the set of such functions. Section~\ref{sec:zerosets} presents a first study of the zero sets. In Section~\ref{sec:quotients}, we construct reciprocals and quotients of slice regular functions. Besides its independent interest, this construction allows us to further the study of zeros in Section~\ref{sec:factorization}, which presents new factorization results. Section~\ref{sec:applications} applies these factorization results both to further study the zero sets and to characterize the points where the real differential of a slice regular function is singular. In Section~\ref{sec:singularities}, we classify the singularities of slice regular functions over slice domains finding that, as it happened with zeros, they may have a peculiar structure when the slice domain is not symmetric. Section~\ref{sec:min} comprises versions of the Minimum Modulus Principle and Open Mapping Theorem for these functions. In Section~\ref{sec:integral}, we prove some integral representation formulas. In order to do so, we prove a stronger version of the Local Extension Theorem, namely~\cite[Theorem 3.2]{localrepresentation}. Section~\ref{sec:sphericalseries} presents a local version of the so-called spherical series expansions for slice regular functions.


\section{Preliminaries}\label{sec:preliminaries}

In this section devoted to preliminaries, we follow the presentation of~\cite[Chapter 1]{librospringer} (which derived from~\cite{cauchy,advancesrevised,open,cras,advances,poli}). The real algebra of quaternions will be denoted as $\hh=\rr+i\rr+j\rr+k\rr$; the real axis as $\rr$; the $2$-sphere of quaternionic imaginary units as $\s$; and the real subalgebra generated by $1$ and by any $I\in\s$ as
\[L_I:=\rr+I\rr\,.\]
If $T \subseteq \hh$, it is customary to set, for each $I \in \s$, the notations $T_I := T \cap L_I$,
\[L_J^+:=\{x+yJ : x,y\in\rr,y>0\}\]
and $T_J^+:=T\cap L_J^+$. As usual, a \emph{domain} in $\hh$ is a nonempty open connected subset of $\hh$.

\begin{definition}\label{sliceregular}
Let $f$ be a quaternion-valued function defined on a domain $\Omega$. For each $I \in \s$, let $f_I := f_{|_{\Omega_I}}$ be the restriction of $f$ to $\Omega_I$. The restriction $f_I$ is called \emph{holomorphic} if it has continuous partial derivatives and
\begin{equation}
\bar \partial_I f(x+yI) := \frac{1}{2} \left( \frac{\partial}{\partial x} + I \frac{\partial}{\partial y} \right) f_I(x+yI) \equiv 0.
\end{equation}
The function $f$ is called \emph{slice regular} if, for all $I \in \s$, $f_I$ is holomorphic.
\end{definition}

For future reference, we recall the following definition, valid for any slice regular function $f:\Omega\to\hh$. After setting
\[\partial_I f(x+yI) := \frac{1}{2} \left( \frac{\partial}{\partial x} - I \frac{\partial}{\partial y} \right) f_I(x+yI)\]
for $x+yI\in\Omega_I$, the \emph{Cullen derivative} (or complex derivative) $f'_c:\Omega\to\hh$ of $f$ is defined as the slice regular function whose restriction $(f'_c)_I$ equals $\partial_I f$ for each $I\in\s$.

\begin{example}
For any $n\in\nn$ and any choice of $a_0,a_1\ldots,a_n$ in $\hh$, the quaternionic polynomial $f(q)=a_0+qa_1+\ldots+q^na_n$ is a slice regular function $\hh\to\hh$. It holds $f'_c(q)=a_1+q2a_2+\ldots+q^{n-1}na_n$.
\end{example}

Power series centered at $0$ provide further examples of slice regular functions, see~\cite[Theorem 1.6]{librospringer}. The analogs at points $p$ other than $0$ have been studied and slice regular functions turned out to admit power series expansions at all points $p$ of their domains, in an appropriate sense, see~\cite[Chapter 2]{librospringer} (derived from~\cite{powerseries}). The latter property is subsumed in Theorem~\ref{laurent} of the present work.

The theory of slice regular functions turned out to be particularly interesting over the class of domains specified in the next definition, because of the subsequent results.

\begin{definition}\label{slicedomain}
Let $\Omega$ be a domain in $\hh$ that intersects the real axis. $\Omega$ is called a \emph{slice domain} if, for all $I \in \s$, the intersection $\Omega_I$ with the complex plane $L_I$ is a domain of $L_I$.
\end{definition}

\begin{theorem}[Identity Principle]\label{identity}
Let $f,g$ be slice regular functions on a slice domain $\Omega$. If, for some $I \in \s$, $f$ and $g$ coincide on a subset of $\Omega_I$ having an accumulation point in $\Omega_I$, then $f = g$ in $\Omega$.
\end{theorem}

\begin{definition}\label{axiallysymmetric}
A set $T \subseteq \hh$ is called \emph{(axially) symmetric} if, for all points $x+yI \in T$ with $x,y \in \rr$ and $I \in \s$, the set $T$ contains the whole sphere $x+y\s$.
\end{definition}

\begin{theorem}[Representation Formula]\label{R-representationformula}
Let $f$ be a slice regular function on a symmetric slice domain $\Omega$ and let $x+y\s \subset \Omega$. For all $I,J,K \in \s$ with $J \neq K$
\begin{eqnarray}
f(x+yI) &=& (J-K)^{-1} \left[J f(x+yJ) - K f(x+yK)\right] +\label{generalrepresentationformula} \\
&+& I (J-K)^{-1} \left[f(x+yJ) - f(x+yK)\right]\,.\nonumber
\end{eqnarray}
Moreover, the quaternion $b := (J-K)^{-1} \left[J f(x+yJ) - K f(x+yK)\right]$ and the quaternion $c := (J-K)^{-1} \left[f(x+yJ) - f(x+yK)\right]$ do not depend on $J,K$ but only on $x,y$.
\end{theorem}

Another version of the Representation Formula was proven in~\cite[Theorem 2.4]{global}, yielding that every slice regular function on a symmetric slice domain is real analytic (see~\cite[Proposition 7]{perotti}). The next two results are also very useful. We quote the second one in the slightly corrected version of~\cite[Theorem 2.1]{localrepresentation}.

\begin{lemma}[Extension Lemma]\label{extensionlemma}
Let $\Omega$ be a symmetric slice domain and let $I \in \s$. If $f_I : \Omega_I \to \hh$ is holomorphic then there exists a unique slice regular function $g : \Omega \to \hh$ such that $g_I = f_I$ in $\Omega_I$. The function $g$ is denoted by $\ext(f_I)$ and called the \emph{regular extension} of $f_I$.
\end{lemma}

\begin{theorem}[Extension Formula]\label{extensionformulathm}
Let $J,K$ be distinct imaginary units; let $T$ be a domain in $L_J$, such that $T_J^+$ is connected and $T \cap \rr \neq \emptyset$; let $U := \{x+yK: x+yJ \in T\}$. Choose holomorphic functions $r : T \to \hh, s : U\to \hh$ such that $r_{|_{T \cap \rr}} = s_{|_{U \cap \rr}}$. Let $\Omega$ be the symmetric slice domain such that $\Omega_J^+=T_J^+, \Omega\cap\rr=T\cap\rr$ and set, for all $x+yI \in \Omega$ with $x,y\in\rr, y\geq0$ and $I\in\s$,
\begin{eqnarray}
f(x+yI) &:=& (J-K)^{-1} \left[J r(x+yJ) - K s(x+yK)\right] + \label{extensionformula}\\
&+& I (J-K)^{-1} \left[r(x+yJ) - s(x+yK)\right] \nonumber
\end{eqnarray}
The function $f: \Omega \to \hh$ is the (unique) slice regular function on $\Omega$ that coincides with $r$ in $\Omega_J^+$, with $s$ in $\Omega_K^+$ and with both $r$ and $s$ in $\Omega\cap\rr$.
\end{theorem}

The study over slice domains that are not symmetric has not been further developed for several years, possibly because it was commonly believed that every slice regular function on a slice domain $\Omega$ could be extended in a unique fashion to the symmetric completion $\widetilde \Omega$ of $\Omega$, in accordance with the next definition.

\begin{definition}
The \emph{(axially) symmetric completion} of a set  $T \subseteq \hh$ is the smallest symmetric set $\widetilde{T}$ that contains $T$. In other words,
\begin{equation}
\widetilde{T} := \bigcup_{x+yI \in T} (x+y\s).
\end{equation}
\end{definition}

However, the recent work~\cite{douren1} proved that not every slice regular function on a slice domain $\Omega$ extends to $\widetilde\Omega$, providing the following counterexample.

\begin{example}\label{ex:douren}
Fix an imaginary unit $I\in\s$. For $J\in\s$, consider the half line
\[h_J:=(-\infty,-2]+2J=\{s+2J : s\in (-\infty,-2]\}\]
and the closed disk
\[D_J:=\{z\in L_J : |z+1-2J|\leq1\}\,.\]
For any $t\in[0,1]$ and all $J\in\s$, consider the following arc, with endpoints $-2+2J$ and $2J$, within $D_J$:
\[\alpha_{t,J}:=-1+2J+\{(1-t)e^{2\pi J s}+t e^{-2\pi J s}: s\in[0,1/2]\}\,.\]
In particular: $\alpha_{0,J}$ is the upper half of the circle $\partial D_J$; $\alpha_{\frac12,J}$ is the line segment $[-2,0]+2J$; and $\alpha_{1,J}$ is the lower half of the circle $\partial D_J$ within $L_J^+$.

For $t\in[0,1]$, consider the unique holomorphic function $\phi_t: L_I\setminus(h_I\cup\alpha_{t,I}) \to L_I$ such that $\phi_t(x+2I)=\ln(x)$ for all $x\in(0,+\infty)$. For all $x,y\in\rr$ such that $x+yI\not\in h_I\cup\alpha_{t,I}$, it holds
\[\phi_t(x+yI):=\ln\sqrt{x^2+(y-2)^2}+I\arg_t(x+I(y-2))\,,\]
where $\arg_t$ denotes the argument on the domain obtained from $L_I$ by erasing both the half line $(-\infty,-2]$ and the arc $\alpha_{t,I}-2I$. Note that $\phi_t$ cannot be holomorphically extended through any point of $\alpha_{t,I}$. For each $t\in(0,1]$, let $C_t$ denote the closed subset of $D_I$ bounded by $\alpha_{0,I}$ and $\alpha_{t,I}$. Then the intersection of the domains of $\phi_t$ and $\phi_0$, namely $L_I\setminus(h_I\cup\alpha_{t,I}\cup\alpha_{0,I})$, has two connected components: the interior of $C_t$ and the set $L_I\setminus (h_I\cup C_t)$. By direct inspection, $\phi_t$ and $\phi_0$ coincide in  $L_I\setminus (h_I\cup C_t)$, while $\phi_t-\phi_0\equiv2\pi I$ in the interior of $C_t$.

If we denote the restriction of $\phi_t$ to $\Phi_t:=L_I \setminus(h_I\cup \overline{h_I}\cup\alpha_{t,I}\cup\overline{\alpha_{t,I}})$ again by $\phi_t$, we may set
\[f_t:=\ext(\phi_t):\widetilde{\Phi}_t\to\hh\,.\]
For each $t\in(0,1]$, the intersection $\widetilde{\Phi}_t\cap\widetilde{\Phi}_0$ has two connected components: the interior of $\widetilde{C}_t$ and $\hh\setminus(\widetilde{h_I}\cup\widetilde{C}_t)$. Clearly, $f_t$ and $f_0$ coincide in the latter connected component. Moreover,
\begin{align*}
f_t(x+Jy)-f_0(x+Jy)&=(2I)^{-1} \left[I2\pi I+I 0\right]+J(2I)^{-1} \left[2\pi I -0\right]\\
&=\pi (I+J)
\end{align*}
for all $x,y\in\rr$ with $y\geq0$ and for all $J\in\s$ such that $x+Jy$ belongs to the interior of $\widetilde{C}_t$. In particular, the difference is zero when $J=-I$.

Let $\Omega$ be the slice domain with the following properties: $\Omega\supset\rr$; moreover, for every $J\in\s$, 
\[\Omega_J^+:=L_J^+\setminus(h_J\cup a_J)\]
where, after setting $T(J):=\min\{|J-I|,1\}$, the arc $a_J$ is defined to equal $\alpha_{T(J),J}$. In particular, $a_I=\alpha_{0,I}$ is the upper half of the circle $\partial D_I$ within $L_I^+$ and, for every imaginary unit $J$ with $|J-I|\geq1$ (including $J=-I)$, $a_J=\alpha_{1,J}$ is the lower half of the circle $\partial D_J$ within $L_J^+$.

Define the slice regular function $f:\Omega\to\hh$ as follows: for each $J\in\s$, $f$ is defined to equal $f_{T(J)}$ in $\Omega_J^+$;  in $\rr$, $f$ equals $f_0$ (whence all $f_t$ with $t\in[0,1]$). Then $f$ admits no regular extension to $\widetilde{\Omega}=\hh\setminus(\widetilde{h_I}\cup\{2I\})$: if it did, then $f_I=(f_0)_I=\phi_0$ would extend holomorphically through $a_I=\alpha_{0,I}$ minus its endpoints; we already noted that this is impossible.
\end{example}

In~\cite{localrepresentation}, we proved the following results, valid on all quaternionic slice domains.

\begin{lemma}\label{gamma}
Let $Y$ be an open subset of $\hh$ and let $J_0\in\s$. Let $C$ be a compact and path-connected subset of $Y_{J_0}$ such that $C \cap \rr$ is a closed interval and $\emptyset\neq C \setminus \rr \subset Y_{J_0}^+$. Let $q_0 \in C$ be such that $\max_{p \in C} |\im(p)| = |\im(q_0)|$. Then there exists $\varepsilon>0$ such that
\[\Gamma(C,\varepsilon) := \bigcup_{p \in C \setminus \rr} B\left(p,\frac{|\im(p)|}{|\im(q_0)|}\varepsilon\right) \cup \bigcup_{p \in C \cap \rr} B(p,\varepsilon)\]
is a slice domain and $C \subset \Gamma(C,\varepsilon) \subseteq Y$.
\end{lemma}

\begin{theorem}[Local Extension]\label{localextension}
Let $f$ be a slice regular function on a slice domain $\Omega$. For every $p_0 \in \Omega$, there exist a symmetric slice domain $N$ with $N \cap \rr \subset \Omega$, a slice domain $\Lambda$ with $p_0 \in \Lambda \subseteq \Omega \cap N$, and a slice regular function $\widetilde f: N \to \hh$ such that $\widetilde f$ coincides with $f$ in $N \cap \rr$, whence in $\Lambda$.
\end{theorem}

We will state and prove a stronger version of this result in Theorem~\ref{localextension2}. Theorem~\ref{localextension} has the following consequence.

\begin{corollary}\label{cor:realanalyticity}
Every slice regular function on a slice domain is real analytic.
\end{corollary}

Motivated by Theorem~\ref{localextension}, we give a new definition.

\begin{definition}
Let $f$ be a slice regular function on a slice domain $\Omega$. If $\widetilde f$ is a slice regular function on a symmetric slice domain $N$ with $N\cap\rr\subset\Omega$ and if there exists a slice domain $\Lambda\subseteq\Omega\cap N$ where $f$ and $\widetilde f$ coincide, then $\big(\widetilde f,N,\Lambda\big)$ is called an \emph{extension triplet} for $f$.
\end{definition}

In~\cite{localrepresentation}, we proved the next result.

\begin{theorem}[Local Representation Formula]\label{localrepresentationthm}
Let $\Omega$ be a slice domain and let $f : \Omega \to \hh$ be a slice regular function. For all $J,K \in \s$ with $J \neq K$ and all $x,y \in \rr$ with $y \geq 0$ such that $x+yJ, x+yK \in \Omega$, let us set
\begin{eqnarray*}
b(x+yJ,x+yK) &:=& (J-K)^{-1} \left[J f(x+yJ) - K f(x+yK)\right]\,,\\
c(x+yJ,x+yK) &:=& (J-K)^{-1} \left[f(x+yJ) - f(x+yK)\right] \,.
\end{eqnarray*}
For every $p_0 \in \Omega$, there exists a slice domain $\Lambda$ with $p_0 \in \Lambda \subseteq \Omega$ such that the following properties hold for all $x,y \in \rr$ with $y \geq 0$:
\begin{itemize}
\item If $U = (x+y\s) \cap \Lambda$ is not empty, then $b,c$ are constant in $(U \times U) \setminus \{ (u,u) : u \in U \setminus \rr\}$.
\item If $I,J,K \in \s$ with $J \neq K$ are such that $x+yI,x+yJ,x+yK \in \Lambda$, then
\begin{equation}
f(x+yI) = b(x+yJ,x+yK) + I c(x+yJ,x+yK) \,. \label{localrepresentationformula}
\end{equation}
\end{itemize}
\end{theorem}

In the previous statement, for $y=0$ we get well-defined $b(x,x), c(x,x)$ because, regardless of the choice of $J,K \in \s$ with $J \neq K$, it holds $(J-K)^{-1} \left[J f(x) - K f(x)\right]=f(x)$ and $(J-K)^{-1} \left[f(x) - f(x)\right] = 0$.

In~\cite{localrepresentation}, we also proved an Extension Theorem valid on a special class of slice domains.

\begin{definition}
A slice domain $\Omega$ is \emph{simple} if, for any choice of $J,K \in \s$, the set
\[\Omega_{J,K}^+:=\{x+yJ \in \Omega_{J}^+: x+yK \in \Omega_{K}^+\}\]
is connected.
\end{definition}

\begin{theorem}[Extension]\label{extension}
Let $f$ be a slice regular function on a simple slice domain $\Omega$. There exists a unique slice regular function $\widetilde f : \widetilde{\Omega} \to \hh$ that extends $f$ to the symmetric completion of its domain.
\end{theorem}

In~\cite{localrepresentation}, we observed that the slice domain $\Omega$ of Example~\ref{ex:douren} is not simple. We also described, without giving a formal proof, a large class of examples of simple domains. We include a proof here. 

\begin{proposition}\label{prop:starlike}
If an open subset $\Omega$ of $\hh$ is starlike with respect to a point $x_0 \in \Omega \cap \rr$, then it is a simple slice domain.
\end{proposition}

\begin{proof}
The slice $\Omega_J$ is starlike with respect to $x_0$ and so is its intersection $C_1$ with the closed half-plane $\overline{L_J^+} = L_J^+ \cup \rr$. Similarly, $\Omega_K \cap \overline{L_K^+}$ is starlike with respect to $x_0$ and so is its ``copy''
\[C_2:=\{x+yJ \in \overline{L_J^+} : x+yK \in \Omega_K \cap \overline{L_K^+}\}\]
within the closed half-plane $\overline{L_J^+}$. Thus, the intersection $C_1 \cap C_2$ is starlike with respect to $x_0$. Moreover, there exists a radius $R>0$ such that the open ball $B(x_0,R)$ is included in $\Omega$. It follows that the half-disk $B(x_0,R) \cap \overline{L_J^+}$ is included in $C_1 \cap C_2$, while the open half-disk $B(x_0,R) \cap L_J^+$ is included in
\[(C_1 \cap C_2) \setminus \rr = \Omega_J^+ \cap (C_2 \setminus \rr) = \Omega_{J,K}^+.\]
We can prove that $\Omega_{J,K}^+$ is path-connected, as follows. For every $p, p' \in \Omega_{J,K}^+$, the line segment between $p$ and $x_0$ and the one between $p'$ and $x_0$ are both included in $C_1 \cap C_2$. Since both line segments enter the open half-disk $B(x_0,R) \cap L_J^+$, it follows that there is a continuous path between $p$ and $p'$ within $\Omega_{J,K}^+$.
\end{proof}


\section{Real differential}\label{sec:differential}

In this section, we begin to study the real differential of a slice regular function on a slice domain $\Omega$. This study will continue in the forthcoming Section~\ref{sec:applications}.

Let $f:\Omega\to\hh$ be a slice regular function. Let us fix $p_0=x_0+y_0I_0 \in \Omega$ (with $x_0,y_0\in\rr, y_0>0$ and $I_0\in\s$) and apply Theorem~\ref{localrepresentationthm}. There exists a neighborhood $U$ of $p_0$ in $x_0+y_0\s$ such that the functions
\[b,c:(U \times U)\setminus\{ (u,u) : u \in U \setminus \rr\}\to\hh\]
are constant. For $p_0$ fixed, their constant values do not depend on the choice of $U$ and can be denoted as $b(p_0),c(p_0)$, respectively. Moreover, for all $p=x_0+y_0I\in U$, Theorem~\ref{localrepresentationthm} implies that
\[f(p)=b(p_0)+Ic(p_0)\]
We must have $b(p)=b(p_0)$ and $c(p)=c(p_0)$.

\begin{definition}\label{def:sphericalvalueandderivative}
Let $\Omega$ be a slice domain in $\hh$ and let $f:\Omega\to\hh$ be a slice regular function. The \emph{spherical value} of $f$ is the function $f^\circ_s:\Omega\to\hh$ mapping each $p_0\in\Omega$ to $b(p_0)$. The \emph{spherical derivative} of $f$ is the function $f'_s:\Omega\setminus\rr\to\hh$ mapping each $p_0=x_0+y_0I_0\in\Omega\setminus\rr$ (with $y_0>0$) to $y_0^{-1}c(p_0)$.
\end{definition}

If $\Omega$ is a symmetric slice domain, then the previous definition is consistent with~\cite[Definition 1.18]{librospringer} (which, in turn, derived from~\cite{perotti}).

The following properties hold by construction. We recall that a slice regular function $f:\Omega\to\hh$ is called \emph{slice preserving} if $f(\Omega_I)\subseteq L_I$ for all $I\in\s$.

\begin{remark}
Let $\Omega$ be a slice domain in $\hh$ and let $f:\Omega\to\hh$ be a slice regular function. By construction, the following properties hold:
\begin{enumerate}
\item $f^\circ_s,f'_s$ are real analytic functions;
\item for each sphere $x_0+y_0\s$ intersecting $\Omega\setminus\rr$, the functions $f^\circ_s,f'_s$ are locally constant in $(x_0+y_0\s)\cap\Omega$;
\item for all $q\in\Omega\setminus\rr$, it holds
\[f(q)=f^\circ_s(q)+\im(q)f'_s(q)\,,\]
while, for all $q\in\Omega\cap\rr$,  it holds $f(q)=f^\circ_s(q)$;
\item $f$ is slice preserving if, and only if, $f^\circ_s,f'_s$ are real-valued.
\end{enumerate}
\end{remark}

\begin{example}
For $p\in\hh\setminus\rr$, consider the binomial $B(q):=q-p$: then $B(\bar q)=\bar q-p$. Thus, in each $x+y\s$, it holds $B^\circ_s\equiv x-p$ and $B'_s\equiv1$.
\end{example}

The next example displays a spherical value and derivative which are locally constant but not constant in $(x_0+y_0\s)\cap\Omega$.

\begin{example}\label{ex:douren2}
Let $f:\Omega\to\hh$ be the slice regular function of Example~\ref{ex:douren}. The $2$-sphere $-1+2\s$ intersects $\Omega$ in two disjoint caps $C^+,C^-$: the former including the point $p:=-1+2I$ of the arc $\alpha_{\frac12,I}$ and the latter including $\bar p$. For all $J\in\s$, it holds:
\begin{align*}
f(-1+2J)&=f_{T(J)}(-1+2J)\\
&=(2I)^{-1} \left[I\phi_{T(J)}(p)+I\phi_{T(J)}(\bar p)\right]+J(2I)^{-1} \left[\phi_{T(J)}(p)-\phi_{T(J)}(\bar p)\right]\\
&=\frac12\left[\phi_{T(J)}(p)+\phi_{T(J)}(\bar p)\right]+\frac{JI}2\left[\phi_{T(J)}(\bar p)-\phi_{T(J)}(p)\right]\,,
\end{align*}
where $\phi_{T(J)}(p)=I\arg_{T(J)}(-1)$ equals $-\pi I$ for $T(J)<\frac12$ and $\pi I$ for $T(J)>\frac12$, while $\phi_{T(J)}(\bar p)=\phi_0(\bar p)=\ln\sqrt{17}+I\arg_0(-1-4I)$ irrespective of $J$. Thus, 
\begin{align*}
(f^\circ_s)_{|_{C^+}}&\equiv\frac12\left[\phi_0(\bar p)-\pi I\right]\,,\\
(f'_s)_{|_{C^+}}&\equiv\frac14\left[I\phi_0(\bar p)-\pi\right]\,,
\end{align*}
while
\begin{align*}
(f^\circ_s)_{|_{C^-}}&\equiv\frac12\left[\phi_0(\bar p)+\pi I\right]\,,\\
(f'_s)_{|_{C^-}}&\equiv\frac14\left[I\phi_0(\bar p)+\pi\right]\,.
\end{align*}
\end{example}

We are now ready to study the real differentials of slice regular functions on slice domains. For the case of symmetric slice domains, see~\cite[Remark 8.15]{librospringer},~\cite[Corollary 6.1]{gporientation}, and~\cite[Proposition 2.3]{conformality}. In the next statement, we use the notion of Cullen derivative $f'_c$, which we recalled in Section~\ref{sec:preliminaries}.

\begin{proposition}\label{prop:differential}
Let $\Omega$ be a slice domain in $\hh$ and let $f:\Omega\to\hh$ be a slice regular function. Pick $p\in\Omega$ and let $df_p$ denote the real differential of $f$ at $p$.
\begin{enumerate}
\item If $p\in\Omega\cap\rr$, then $df_p(v)=vf'_c(p)$ for all $v\in T_p\Omega\simeq\hh$. As a consequence, $df_p$ is singular if, and only if, $f'_c(p)=0$.
\item If $p\in\Omega\setminus\rr$, if $I\in\s$ is such that $p\in\Omega_I$ and if we split $T_p\Omega\simeq\hh$ as $L_I\oplus L_I^\perp$, then
\[df_p(v+w)=vf'_c(p)+wf'_s(p)\]
for all $v\in L_I$ and $w\in L_I^\perp$. As a consequence, $df_p$ is singular if, and only if, $f'_c(p)\overline{f'_s(p)}\in L_I^\perp$.
\end{enumerate}
\end{proposition}

\begin{proof}
By Corollary~\ref{cor:realanalyticity}, $f$ is real differentiable. Let us compute the real differential $df_p$ at $p\in\Omega$.
If $p=x+yI$ for some $x,y\in\rr$ and $I\in\s$, then the definition of slice regular function implies that $df_p(v)=vf'_c(p)$ for all $v\in L_I$. We separate two cases:
\begin{enumerate}
\item If $p\in\Omega\cap\rr$, it follows that $df_p(v)=vf'_c(p)$ for all $v\in T_p\Omega\simeq\hh$. Clearly, $df_p$ is singular if, and only if, $f'_c(p)=0$.
\item If $p\in\Omega\setminus\rr$, we first prove that $df_p(w)=wf'_s(p)$ for all $w\in L_I^\perp$. This follows from property {\it 3.} in the previous remark, if we take into account that $L_I^\perp$ is the tangent space to the $2$-sphere $x+y\s$ at $p$. Secondly, we observe that the image of $df_p$, which is $L_If'_c(p)+L_I^\perp f'_s(p)$, has dimension $4$ if, and only if, $L_If'_c(p)\overline{f'_s(p)}+L_I^\perp$ does, which is the same as $f'_c(p)\overline{f'_s(p)}\not\in L_I^\perp$.
\qedhere
\end{enumerate} 
\end{proof}

Although this is not strictly needed for the present work, we can generalize the notion of \emph{slice function} introduced in~\cite[Definition 5]{perotti} with the following definition and make the subsequent remarks.

\begin{definition}
Let $Y\subseteq\hh$. Then $f:Y\to\hh$ is a \emph{locally slice function} if there exist functions $b,c:Y\to\hh$ such that, for all $S:=x_0+y_0\s$ (with $x_0,y_0\in\rr$ and $y_0\geq0$), the following properties hold:
\begin{enumerate}
\item $b,c$ are constant in each connected component of $S\cap Y$;
\item for all $I\in\s$ such that $x_0+y_0I\in Y$, it holds
\[f(x_0+y_0I)=b(x_0+y_0I)+Ic(x_0+y_0I)\,.\]
\end{enumerate}
If this is the case, the \emph{spherical value} $f^\circ_s:Y\to\hh$ and \emph{spherical derivative} $f'_s:Y\setminus\rr\to\hh$ are defined by setting $f^\circ_s(x_0+y_0I):=b(x_0+y_0I)$ if $x_0+y_0I\in Y$ and $f'_s(x_0+y_0I):=y_0^{-1}c(x_0+y_0I)$ if $x_0+y_0I\in Y\setminus\rr$.
\end{definition}

If $f$ is locally slice then, by construction, $f(q)=f^\circ_s(q)+\im(q)f'_s(q)$ for all $q\in Y\setminus\rr$ and $f(q)=f^\circ_s(q)$ for all $q\in Y\cap\rr$. The function $f$ is \emph{slice preserving}, i.e., $f(Y_I)\subseteq L_I$ for all $I\in\s$, if, and only if $f^\circ_s,f'_s$ are real-valued.

Of course, all slice regular functions on slice domains are locally slice functions. It is possible to produce other examples of locally slice functions by considering linear combinations of slice regular functions that are defined on slice domains whose intersection $Y$ is not a slice domain. This is the case with the difference $D:=f_1-f_0$ between the two slice regular functions $f_0$ and $f_1$ defined in Example~\ref{ex:douren}, which is a slice function on the union between the interior of the solid torus $\widetilde{C}_1$ and the open set $\hh\setminus(\widetilde{h_I}\cup\widetilde{C}_1)$.


\section{Algebraic structure}\label{sec:algebraic}

In this section, we endow the set of slice regular functions on a slice domain $\Omega$ with a $^*$-algebra structure. We will follow the approach adopted in~\cite{perotti} over symmetric slice domains. This approach is distinct from the one of~\cite[\S 1.4]{librospringer} (which derived from~\cite{advancesrevised,zeros}).

\begin{definition}\label{def:operations}
Let $\Omega$ be a slice domain in $\hh$ and let $f,g:\Omega\to\hh$ be slice regular functions. We define the \emph{regular product} $f*g:\Omega\to\hh$ by setting
\[f*g(q):=f^\circ_s(q)\,g^\circ_s(q)+\im(q)^2 f'_s(q)\, g'_s(q) + \im(q)(f^\circ_s(q)\, g'_s(q)+ f'_s(q)\, g^\circ_s(q))\]
for all $q\in\Omega\setminus\rr$ and $f*g(q):=f^\circ_s(q)\,g^\circ_s(q)=f(q)g(q)$ for all $q\in\Omega\cap\rr$.

We define the \emph{regular conjugate} $f^c:\Omega\to\hh$ by setting
\[f^c(q):=\overline{f^\circ_s(q)}+\im(q) \overline{f'_s(q)}\]
for all $q\in\Omega\setminus\rr$ and $f^c(q):=\overline{f^\circ_s(q)}=\overline{f(q)}$ for all $q\in\Omega\cap\rr$.

The \emph{symmetrization} $f^s:\Omega\to\hh$ is defined as
\[f^s(q)=|f^\circ_s(q)|^2-|\im(q)f'_s(q)|^2+2\im(q)\re(f^\circ_s(q)\, \overline{f'_s(q)})\,,\]
for all $q\in\Omega\setminus\rr$ and as $f^s(q)=|f^\circ_s(q)|^2=|f(q)|^2$ for all $q\in\Omega\cap\rr$.
\end{definition}

If $\Omega$ happens to be a symmetric slice domain, then the previous definition is consistent with~\cite[Definitions 1.27, 1.34, 1.35]{librospringer} by~\cite[Remark 3.2]{gpsalgebra}. This is true, in particular, for quaternionic polynomials:

\begin{example}
For $p=x_0+y_0I\in\hh$, consider the binomial $B(q)=q-p$. It holds 
\begin{align*}
B^c(q)&=q-\bar p\,,\\
B^s(q)&=q^2-q2\re(p)+|p|^2=(q-x_0)^2+y_0^2\,.
\end{align*}
In each $x+y\s$ with $x,y\in\rr$ and $y>0$, it holds
\begin{align*}
&(B^c)^\circ_s\equiv x-\bar p=x-x_0+y_0I,\quad (B^c)'_s\equiv1\,,\\
&(B^s)^\circ_s\equiv|x-x_0|^2+y_0^2-y^2,\quad (B^s)'_s\equiv2(x-x_0)\,.
\end{align*}
\end{example}

The set of slice regular functions on a symmetric slice domain is a real $^*$-algebra with the operations $+,*,^c$ (see, for instance,~\cite[page 4741]{gpsalgebra}). We extend this property to all slice domains.

\begin{theorem}\label{thm:operations}
Let $\Omega$ be a slice domain in $\hh$. Then $+,*,^c$ are operations on the set of slice regular functions on $\Omega$ and turn it into a real $^*$-algebra. Moreover, if $f,g:\Omega\to\hh$ are slice regular functions, then the following properties hold:
\begin{enumerate}
\item if $f$ is slice preserving, then $f*g(q)=f(q)g(q)=g*f(q)$ in $\Omega$;
\item if $g$ is constantly equal to $c$, then $f*g(q)=f(q)c$ for all $q\in\Omega$;
\item $f^s$ is a slice preserving regular function and \[f^s=f*f^c=f^c*f\,.\]
\end{enumerate}
\end{theorem}

\begin{proof}
Let $f,g:\Omega\to\hh$ be slice regular functions. By direct inspection in Definition~\ref{sliceregular}, pointwise addition (i.e., setting $(f+g)(q):=f(q)+g(q)$ for all $q\in\Omega$) produces a slice regular function $f+g:\Omega\to\hh$. Now let us look at the regular conjugate $f^c$ and product $f*g$.

We apply Theorem~\ref{localextension} to find, for every point $p \in \Omega$, an extension triplet $\big(\widetilde f,N,\Lambda\big)$ for $f$ with $p\in\Lambda$. The regular conjugate of $\widetilde f$ is a slice regular function $h:N\to\hh$ with
\[h(q)=\overline{\widetilde f^\circ_s(q)}+\im(q) \overline{\widetilde f'_s(q)}\]
for all $q\in N\setminus\rr$ and $h(q)=\overline{\widetilde f^\circ_s(q)}=\overline{\widetilde f(q)}$ for all $q\in N\cap\rr$. Since $\widetilde f^\circ_s(q)=f^\circ_s(q)$ for all $q\in\Lambda$ and $\widetilde f'_s(q)=f'_s(q)$ for all $q\in\Lambda\setminus\rr$, we conclude that $h(q)=f^c(q)$ for all $q\in\Lambda$. Since $p$ was arbitrarily chosen in $\Omega$, it follows that $f^c$ is a slice regular function on $\Omega$.

Now let us apply Theorem~\ref{localextension} to find an extension triplet $\big(\widetilde g,N',\Lambda'\big)$ for the slice regular function $g_{|_{\Lambda}}:\Lambda\to\hh$, with $p\in\Lambda'$. The intersection $N\cap N'$ is a symmetric open set with
\[N\cap N'\supseteq \Lambda \cap N'\supseteq \Lambda'\ni p\,.\]
Since $\Lambda'$ is a slice domain, the connected component of $N\cap N'$ that includes $\Lambda'$ is a symmetric slice domain $M$ that includes $p$. The product $\widetilde f_{|_{M}}*\widetilde g_{|_{M}}$ is a slice regular function $h:M\to\hh$ with
\[h(q)=\widetilde f^\circ_s(q)\,\widetilde g^\circ_s(q)+\im(q)^2 \widetilde f'_s(q)\, \widetilde g'_s(q) + \im(q)(\widetilde f^\circ_s(q)\, \widetilde g'_s(q)+ \widetilde f'_s(q)\, \widetilde g^\circ_s(q))\]
for all $q\in M\setminus\rr$ and $h(q)=\widetilde f^\circ_s(q)\,\widetilde g^\circ_s(q)$ for all $q\in M\cap\rr$. Since $\widetilde f^\circ_s(q)=f^\circ_s(q),\widetilde g^\circ_s(q)=g^\circ_s(q)$ for all $q\in\Lambda'$ and $\widetilde f'_s(q)=f'_s(q),\widetilde g'_s(q)=g'_s(q)$ for all $q\in\Lambda'\setminus\rr$, we conclude that $h(q)=f*g(q)$ for all $q\in\Lambda'$. Again, it follows that $f*g$ is a slice regular function on $\Omega$.

Now let us prove that the set of slice regular functions on $\Omega$ is a real algebra. It clearly is a real vector space. We are left with proving that regular multiplication $(f,g)\mapsto f*g$ is bilinear and that it is associative. Both properties can be derived either from the definition or from the corresponding properties of slice regular functions over symmetric slice domains (using the technique applied in first part of the proof). Moreover, the set of slice regular functions on $\Omega$ is a real algebra: the definition immediately implies that $(f^c)^c=f$ for all $f$ and $f^c=f$ for $f\equiv a\in\rr$; the equality $(f*g)^c=g^c*f^c$ can be derived either from the definition or from the corresponding property of slice regular functions over symmetric slice domains.

Now let us prove the last statement. Properties {\it 1.} and {\it 2.} are proven by direct inspection in Definition~\ref{def:operations}. The equality in property {\it 3.} follows by direct computation from Definition~\ref{def:operations} and yields that $f^s$ is slice regular. By direct inspection in Definition~\ref{def:operations}, the spherical value and derivative of $f^s$ are real-valued, whence $f^s$ is a slice preserving function.
\end{proof}

As a byproduct of the previous proof, we can prove what follows.

\begin{proposition}\label{prop:productformula}
Let $\Omega$ be a slice domain in $\hh$, let $f:\Omega\to\hh$ be a slice regular function and pick $p=x+yI\in\Omega$. If $f(p)=0$, then $f*g(p)=0$. Otherwise, the point $f(p)^{-1}pf(p)$ belongs to $x+y\s$ and
\[f*g(p)=f(p)\,\widetilde g(f(p)^{-1}pf(p))\]
where $\widetilde g(q)=g^\circ_s(p)+\im(q)g'_s(p)$ for all $q\in x+y\s$ (whence $\widetilde g$ coincides with $g$ in the connected component of $(x+y\s)\cap\Omega$ that includes $p$). 
\end{proposition}

\begin{proof}
In the previous proof, we had found an extension triplet $\big(\widetilde f,N,\Lambda\big)$ for $f$ with $p\in\Lambda$ and an extension triplet $\big(\widetilde g,N',\Lambda'\big)$ for $g_{|_{\Lambda}}:\Lambda\to\hh$ with $p\in\Lambda'$. We had also found a symmetric slice domain $M$ with $p\in\Lambda'\subseteq M\subseteq N\cap N'$, such that the functions $f*g$ and $\widetilde f_{|_{M}}*\widetilde g_{|_{M}}$ coincided in $\Lambda'$. In particular, $f*g(p)=\widetilde f_{|_{M}}*\widetilde g_{|_{M}}(p)$.

Let us apply~\cite[Theorem 3.4]{librospringer} to $\widetilde f_{|_{M}},\widetilde g_{|_{M}}$: if $\widetilde f(p)=0$, then $\widetilde f_{|_{M}}*\widetilde g_{|_{M}}(p)=0$; otherwise,
\[\widetilde f_{|_{M}}*\widetilde g_{|_{M}}(p)=\widetilde f(p)\, \widetilde g\left(\widetilde f(p)^{-1}p\widetilde f(p)\right)\]
Remembering that $\widetilde f(p)=f(p)$, we derive the first two equalities appearing in our statement. Moreover, by construction, for all $q\in x+y\s$ it holds $\widetilde g(q)=g^\circ_s(p)+\im(q)g'_s(p)$. The latter expression coincides with $g(q)$ for all $q$ in the connected component of $(x+y\s)\cap\Omega$ that includes $p$.
\end{proof}

We conclude this section with the following remarks, though not essential for the present work. If we modify Definition~\ref{def:operations} assuming $f,g$ to be locally slice functions on $Y\subseteq\hh$, we get well-defined locally slice functions $f*g,f^c,f^s:Y\to\hh$, which we may call \emph{slice product} of $f$ and $g$, \emph{slice conjugate} of $f$ and \emph{normal function} of $f$. This generalizes~\cite[Definitions 9 and 11]{perotti}. Analogs of Theorem~\ref{thm:operations} and Proposition~\ref{prop:productformula} hold.


\section{Zero sets}\label{sec:zerosets}

This section studies the zero sets
\[Z(f):=\{q\in\Omega\,:\,f(q)=0\}\]
of slice regular functions $f$ on slice domains $\Omega$. For the case of symmetric slice domains, see~\cite[Chapter 3]{librospringer} (which, in turn, derived from~\cite{advancesrevised,zeros,zerosopen,advances}). If $\Omega$ is symmetric, then $Z(f)$ consists of isolated points or isolated $2$-spheres of type $x+y\s$ (unless it is the whole domain). Moreover, each $2$-sphere of type $x+y\s$ cannot include several isolated zeros. In the case of slice domains, the situation is more manifold. We will see theoretical results here and explicit examples in the forthcoming Section~\ref{sec:factorization}.

\begin{proposition}\label{prop:zeros}
Let $\Omega$ be a slice domain in $\hh$, let $f:\Omega\to\hh$ be a slice regular function and consider its zero set $Z(f)$. Fix any $2$-sphere $S=x+y\s$ (with $x,y\in\rr,y\neq0$). For $p\in S$, the equality $f(p)=0$ is equivalent to 
\[f^\circ_s(p)=-\im(p)f'_s(p)\,.\]
If the previous equality holds and if $C$ denotes the connected component of $\Omega\cap S$ that includes $p$, then either
\begin{enumerate}
\item $f'_s(p)\neq0$ and $Z(f)\cap C=\{p\}$; or
\item $f'_s(p)=0$ and $C\subseteq Z(f)$.
\end{enumerate}
The former possibility is excluded if $f$ is slice preserving.
\end{proposition}

\begin{proof}
Since $f^\circ_s,f'_s$ are locally constant in $S\cap\Omega$, they are constant in $C$. Thus, for all $q\in C$,
\[f(q)=f^\circ_s(q)+\im(q)f'_s(q)=f^\circ_s(p)+\im(q)f'_s(p)\,.\]
Moreover, $0=f(p)=f^\circ_s(p)+\im(p)f'_s(p)$ is equivalent to $f^\circ_s(p)=-\im(p)f'_s(p)$. In such a case,
\[f(q)=(\im(q)-\im(p))f'_s(p)\]
for all $q\in C$. If $f'_s(p)=0$, then $f$ vanishes identically in $C$. Otherwise, $p$ is the only zero of $f$ in $C$.
Now let us address the case of a slice preserving $f$: since $f^\circ_s,f'_s$ are real-valued, the equality $f^\circ_s(p)=-\im(p)f'_s(p)$ implies $f^\circ_s(p)=0=f'_s(p)$.
\end{proof}

\begin{theorem}\label{thm:zeros}
Let $\Omega$ be a slice domain in $\hh$ and let $f:\Omega\to\hh$ be a slice regular function. Either $f\equiv0$ or, for any $S=x+y\s$ intersecting $\Omega$ and for any connected component $C$ of $S\cap\Omega$, the set $Z(f)\setminus C$ has no accumulation points in $C$.
\end{theorem}

\begin{proof}
Let us suppose $C$ to include an accumulation point $p$ of the set $Z(f)\setminus C$: in other words, there is a sequence 
$\{p_n\}_{n\in\nn}\subseteq Z(f)\setminus C$ converging to $p$. Let us prove that $f\equiv0$.

By Theorem~\ref{localextension}, $f$ has an extension triplet $\big(\widetilde f,N,\Lambda\big)$ with $p\in\Lambda$. Up to refinements, it holds $\{p_n\}_{n\in\nn}\subseteq\Lambda$, whence $\widetilde f(p_n)=f(p_n)=0$ for all $n\in\nn$. Up to a further refinement, it also holds $p_n\not\in S$: if such a refinement were not possible, then the limit point $p$ would belong to $(S\cap\Omega)\setminus C$, a contradiction. Thus, $p$ is an accumulation point for $Z(\widetilde f)\setminus S$. By~\cite[Theorem 3.12]{librospringer}, $\widetilde f\equiv0$.

It follows that $f_{|_\Lambda}\equiv0$. By the Identity Principle~\ref{identity}, the function $f$ vanishes identically in $\Omega$.
\end{proof}

Overall, if $f\not\equiv0$, then $Z(f)$ consists of isolated points or whole connected components of $(x+y\s)\cap\Omega$, each isolated from the rest of $Z(f)$. More precisely:

\begin{corollary}\label{cor:zeros}
Let $\Omega$ be a slice domain in $\hh$ and let $f:\Omega\to\hh$ be a slice regular function with $f\not\equiv0$. Suppose $p\in Z(f)$, let $S=x+y\s$ include $p$ and let $C$ be connected component of $S\cap\Omega$ that includes $p$. If either $p\in\rr$ or $f'_s(p)\neq0$, then $p$ is an isolated point in $Z(f)$. If, instead, $p\not\in\rr$ and $f'_s(p)=0$, then $Z(f)$ includes $C$ and $Z(f)\setminus C$ has no accumulation points in $C$.
\end{corollary}

\begin{proof}
If either $p\in\rr$ or $f'_s(p)\neq0$, then $Z(f)\cap C=\{p\}$ by Proposition~\ref{prop:zeros}. It follows that $Z(f)\setminus C=Z(f)\setminus\{p\}$. If, instead, $p\not\in\rr$ and $f'_s(p)=0$, then Proposition~\ref{prop:zeros} yields that $Z(f)\supseteq C$. By Theorem~\ref{thm:zeros}, in all cases, $Z(f)\setminus C$ has no accumulation points in $C$. The thesis follows.
\end{proof}

One novelty with respect to the symmetric case is, that $(x+y\s)\cap\Omega$ may include several zeros of $f$ without being entirely included in $Z(f)$. Explicit examples will be provided in Section~\ref{sec:factorization}.

Let us study the zero set of a regular product $f*g$ and find that it consists of zeros of $f$, plus points that depend on $g$ but are not necessarily zeros of $g$.

\begin{proposition}\label{prop:zerosofproduct}
Let $\Omega$ be a slice domain in $\hh$ and let $f,g:\Omega\to\hh$ be slice regular functions. Then
\[Z(f)\subseteq Z(f*g)\,.\]
Moreover,
\[Z(f*g)\setminus Z(f)=\{p\in\Omega : g^\circ_s(p)=-f(p)^{-1}\im(p)f(p)g'_s(p)\}\,.\]
\end{proposition}

\begin{proof}
Proposition~\ref{prop:productformula} immediately yields the first statement $Z(f)\subseteq Z(f*g)$. Moreover, it implies that a point $p\not\in Z(f)$ belongs to $Z(f*g)$ if, and only if, $f(p)^{-1}pf(p)$ is a zero of a function $\widetilde g$ such that $\widetilde g(q)=g^\circ_s(p)+\im(q)g'_s(p)$ for all $q\in x+y\s$. We now remark that
\[\im\left(f(p)^{-1}pf(p)\right)=f(p)^{-1}\im(p)f(p)\,,\]
whence our second statement follows.
\end{proof}

We will see in Section~\ref{sec:factorization} that $Z(f*g)\setminus Z(f)$ may be non-empty even when $Z(g)$ is empty. Definition~\ref{def:operations} allows the following remark about $Z(f^c)$.

\begin{remark}\label{rmk:zerosofconjugate}
Let $\Omega$ be a slice domain in $\hh$, let $f:\Omega\to\hh$ be a slice regular function and let $p\in\Omega$. Then $p\in Z(f^c)$ if, and only if,
\[f^\circ_s(p)=f'_s(p)\im(p)\,.\]
\end{remark}

Let us now turn to the study of $Z(f^s)$.

\begin{proposition}\label{prop:zerosofsymmetrization1}
Let $\Omega$ be a slice domain in $\hh$ and let $f:\Omega\to\hh$ be a slice regular function. The zero sets $Z(f)$ and $Z(f^s)$ are related, as follows:
\begin{enumerate}
\item Fix any $2$-sphere $S=x+y\s$ (with $x,y\in\rr,y\neq0$). If $p\in Z(f)$ and if $C$ denotes the connected component of $\Omega\cap S$ that includes $p$, then $C\subseteq Z(f^s)$.
\item The equality $Z(f)\cap\rr=Z(f^s)\cap\rr$ holds. Thus, if $f^s\equiv0$, then $f\equiv0$.
\end{enumerate}
\end{proposition}

\begin{proof}
\begin{enumerate}
\item Since $f^s=f*f^c$, Proposition~\ref{prop:zerosofproduct} implies that $Z(f)\subseteq Z(f^s)$. Since $f^s$ is slice preserving, by Proposition~\ref{prop:zeros}, if $p\in Z(f^s)$ and if $C$ denotes the connected component of $\Omega\cap S$ that includes $p$, then $C\subseteq Z(f^s)$.
\item If $q\in\Omega\cap\rr$, then $f^s(q)=|f(q)|^2$ vanishes if, and only if, $f(q)=0$. Now, if $f^s\equiv0$, then $\Omega\cap\rr\subseteq Z(f)$. By the Identity Principle~\ref{identity}, it follows that $f\equiv0$.\qedhere
\end{enumerate}
\end{proof}

For future reference, we make the following remark.

\begin{remark}
Let $\Omega$ be a slice domain in $\hh$. If $f:\Omega\to\hh$ is a slice regular function with $f\not\equiv0$, then we can combine several results in this section to conclude that $Z(f^s)$ consists of: isolated real points; or whole connected components of $(x+y\s)\cap\Omega$, each isolated from the rest of $Z(f^s)$. As a consequence, $\Omega\setminus Z(f^s)$ is a slice domain.
\end{remark}

Although not strictly needed for the present work, we can make the following remarks. Proposition~\ref{prop:zeros}, Proposition~\ref{prop:zerosofproduct} and Remark~\ref{rmk:zerosofconjugate} still hold true if we relax the assumptions on $f,g$, assuming them to be locally slice functions on an arbitrary subset $Y$ of $\hh$. Under the same assumptions, property {\it 1.} in Proposition~\ref{prop:zerosofsymmetrization1} holds true, while property {\it 2.} is reduced to the equality $Z(f)\cap\rr=Z(f^s)\cap\rr$. Indeed, it may well happen that $f^s\equiv0$ while $f\not\equiv0$:

\begin{example}\label{ex:zerodivisor}
Consider the slice regular functions $f_0$ and $f_1$ defined in Example~\ref{ex:douren} and their difference $D:=f_1-f_0$ on the interior of the solid torus $\widetilde{C}_1$. For $x,y\in\rr$ with $y\geq0$ and $J\in\s$ such that $x+Jy$ belongs to the interior of $\widetilde{C}_1$, it holds $D(x+Jy)=\pi(I+J)=\pi I+yJ\frac{\pi}{y}$. Thus, $D^c(x+Jy)=\pi I-yJ\frac{\pi}{y}$ and
\[D^s(x+Jy)=\pi^2-y^2\frac{\pi^2}{y^2}+2yJ\re\left(-\frac{\pi^2}{y}I\right)\equiv0\,.\]
\end{example}


\section{Quotients}\label{sec:quotients}

This section studies division between slice regular functions on a slice domain $\Omega$. We begin by defining the regular reciprocal of such a function $f:\Omega\to\hh$. If $\Omega$ happens to be a symmetric slice domain, then the new definition we give is consistent with~\cite[Definition 5.1]{librospringer} by~\cite[Theorem 4.4]{gpsdivisionalgebras}.

\begin{definition}\label{def:reciprocal}
For all $a,b\in\hh$ with $|a|\neq|b|$ or $\re(a\bar b)\neq0$, set
\[\Phi(a,b) := \frac{|a|^2\bar a+\bar ba\bar b}{(|a|^2-|b|^2)^2+(2\re(a\bar b))^2}\,.\]
Let $\Omega$ be a slice domain in $\hh$. If $f:\Omega\to\hh$ is a slice regular function with $f\not\equiv0$, consider the slice domain $\Omega':=\Omega\setminus Z(f^s)$. The \emph{regular reciprocal} $f^{-*}:\Omega'\to\hh$ of $f$ is defined as follows. 
\[f^{-*}(q):=\Phi\big(f^\circ_s(q),|\im(q)|\,f'_s(q)\big)-\ \frac{\im(q)}{|\im(q)|}\ \Phi\big(|\im(q)|\,f'_s(q),f^\circ_s(q)\big)\]
for all $q\in\Omega'\setminus\rr$ and $f^{-*}(q):=f^\circ_s(q)^{-1}=f(q)^{-1}$ for all $q\in\Omega'\cap\rr$.
\end{definition}

The last displayed formula is well defined if, and only if, $|f^\circ_s(q)|\neq|\im(q)f'_s(q)|$ or $\re(f^\circ_s(q)\, f'_s(q)^c)\neq0$, which is equivalent to $f^s(q)\neq0$ by Definition~\ref{def:operations}. This is the reason why the domain of $f^{-*}$ is $\Omega':=\Omega\setminus Z(f^s)$. We can now prove the following result.

\begin{proposition}\label{prop:quotient}
Let $\Omega$ be a slice domain in $\hh$ and let $f:\Omega\to\hh$ be a slice regular function such that $f\not\equiv0$. Then $f^{-*}$ is a slice regular function on $\Omega'=\Omega\setminus Z(f^s)$ and the following equalities hold in $\Omega'$: 
\begin{align*}
f^{-*}(q)&=(f^s)^{-*}*f^c(q)=f^s(q)^{-1}f^c(q)\,,\\
f*f^{-*}&=f^{-*}*f\equiv1\,.
\end{align*}
Moreover, if $f$ is slice preserving, then $f^{-*}$ is slice preserving and $f^{-*}(q)=f(q)^{-1}$ for all $q\in\Omega'$.
\end{proposition}

\begin{proof}
Let us fix $p\in\Omega'$. By Theorem~\ref{localextension}, $f$ has an extension triplet $\big(\widetilde f,N,\Lambda\big)$ with $p\in\Lambda$.

If we set $N':=N\setminus Z\big(\widetilde f^s\big)$ and if we define $g:N'\to\hh$ by setting $g(q):=\widetilde f^s(q)^{-1}\widetilde f^c(q)$, then $g$ is the left and right multiplicative inverse of $\widetilde f_{|_{N'}}$ by the theory presented in~\cite[\S 5.1]{librospringer}. Moreover, by~\cite[Theorem 4.4]{gpsdivisionalgebras} it holds 
\[g(q) = \Phi\left(\widetilde f^\circ_s(q),|\im(q)|\,\widetilde f'_s(q)\right)-\ \frac{\im(q)}{|\im(q)|}\ \Phi\left(|\im(q)|\,\widetilde f'_s(q),\widetilde f^\circ_s(q)\right)\]
in $N'\setminus\rr$ and $g(q)=\widetilde f^\circ_s(q)^{-1}=\widetilde f(q)^{-1}$ for all $q\in N'\cap\rr$.
Now, as a consequence of the last equalities, in 
\[\Lambda\setminus Z\big(\widetilde f^s\big)=\Lambda\setminus Z(f^s)=\Lambda\cap\Omega'\]
it holds
\[f^{-*}(q)=g(q)=\widetilde f^s(q)^{-1}\widetilde f^c(q)=f^s(q)^{-1}f^c(q)\,.\]

Because $p$ was arbitrarily chosen in $\Omega'$, it follows that $f^{-*}:\Omega'\to\hh$ is a slice regular function fulfilling the equality $f^{-*}(q)=f^s(q)^{-1}f^c(q)$ for all $q\in\Omega'$. Moreover, for all $q\in\Omega'\cap\rr$,
\begin{align*}
&f^{-*}(q) =f^s(q)^{-1}f^c(q)=(f^c(q)f(q))^{-1}f^c(q)=f(q)^{-1}\\
&f^{-*}*f(q)=f(q)^{-1}f(q)=1\\
&f*f^{-*}(q)=f(q)f(q)^{-1}=1\,.
\end{align*}
Since $\Omega'$ is a slice domain, by the Identity Principle~\ref{identity} it follows that
\[f^{-*}*f=f*f^{-*}\equiv1\]
 in $\Omega'$. In the special case when $f$ is slice preserving, the equality $1=f*f^{-*}(q)=f(q)f^{-*}(q)$ valid for all $q\in\Omega'$ implies that $f^{-*}(q)=f(q)^{-1}$ for all $q\in\Omega'$, whence $f^{-*}$ is slice preserving too.

Finally, let us prove that $f^{-*}=(f^s)^{-*}*f^c$. Because $f^s$ is always a slice preserving function, we now know that $(f^s)^{-*}$ is a slice preserving regular function on $\Omega'$ with $(f^s)^{-*}(q)=f^s(q)^{-1}$ for all $q\in\Omega'$. Thus,
\[f^{-*}(q)=f^s(q)^{-1}f^c(q)=(f^s)^{-*}(q)f^c(q)=(f^s)^{-*}*f^c(q)\]
for all $q\in\Omega'$, as desired.
\end{proof}

\begin{example}
For $p=x_0+y_0I\in\hh$, consider the binomial $B(q)=q-p$. It holds 
\[B^{-*}(q)=[(q-x_0)^2+y_0^2]^{-1}(q-\bar p)\,.\]
In each $x+y\s$ with $x,y\in\rr$ and $y>0$, it holds
\begin{align*}
&(B^{-*})^\circ_s\equiv\Phi(x-p,y)\,,\\
&(B^{-*})'_s\equiv-y^{-1}\Phi(y,x-p)\,.
\end{align*}
\end{example}

Proposition~\ref{prop:quotient} allows us to make the following remark.

\begin{remark}
The algebra of slice regular functions on a slice domain $\Omega$ admits no zero divisors. Indeed, consider two elements $f,g$ with $f\equiv0$, so that $f^{-*}$ is defined on the slice domain $\Omega'=\Omega\setminus Z(f^s)$: if $f*g\equiv0$, then $g=f^{-*}*f*g\equiv0$ in $\Omega'$; whence $g\equiv0$ in $\Omega$ by the Identity Principle~\ref{identity}.
\end{remark}

We now give the definition of regular quotient and find an explicit expression for quotients.

\begin{definition}\label{def:quotient}
Let $\Omega$ be a slice domain in $\hh$ and let $f,g:\Omega\to\hh$ be slice regular functions. We define the \emph{regular quotient} of $f$ and $g$ as $f^{-*}*g:\Omega\setminus Z(f^s)\to\hh$.
\end{definition}

\begin{proposition}\label{prop:quotientformula}
Let $\Omega$ be a slice domain in $\hh$ and let $f,g:\Omega\to\hh$ be slice regular functions. For each $p=x+yI\in\Omega\setminus Z(f^s)$, it holds
\[f^{-*}*g(p)=\widetilde f(T_f(p))^{-1}\,\widetilde g(T_f(p))\]
where $T_f(p):=f^c(p)^{-1}pf^c(p)$, while $\widetilde f(q)=f^\circ_s(p)+\im(q)f'_s(p)$ and $\widetilde g(q)=g^\circ_s(p)+\im(q)g'_s(p)$ for all $q\in x+y\s$.
\end{proposition}

\begin{proof}
According to Proposition~\ref{prop:quotient},
\[f^{-*}*g(p)=(f^s)^{-*}*f^c*g(p)=f^s(p)^{-1}(f^c*g(p))\,.\]
We remark that $f^c(p)\neq0$: this follows by applying Proposition~\ref{prop:zerosofsymmetrization1} to $(f^c)^s=f^s$ because we have assumed $f^s(p)\neq0$. By Proposition~\ref{prop:productformula}, it holds
\begin{align*}
&f^s(p)=f^c*f(p)=f^c(p)\widetilde f(T_f(p))\\
&f^c*g(p)=f^c(p)\widetilde g(T_f(p))\,.
\end{align*}
Thus,
\[f^{-*}*g(p)=\widetilde f(T_f(p))^{-1}f^c(p)^{-1}f^c(p)\widetilde g(T_f(p))=\widetilde f(T_f(p))^{-1}\,\widetilde g(T_f(p))\,,\]
as desired.
\end{proof}

We conclude this section with the following remarks. If we modify Definition~\ref{def:reciprocal} and Definition~\ref{def:quotient} assuming $f,g$ to be locally slice functions on $Y\subseteq\hh$ and $f^s\not\equiv0$, we get well-defined locally slice functions $f^{-*},f^{-*}*g:Y\setminus Z(f^s)\to\hh$, which we may call \emph{slice reciprocal} of $f$ and \emph{slice quotient} of $f$ and $g$. Analogs of Proposition~\ref{prop:quotient} and Proposition~\ref{prop:quotientformula} hold. We point out that, as we saw in Example~\ref{ex:zerodivisor}, it is not enough to assume $f\not\equiv0$ to guarantee $f^s\not\equiv0$; in particular, there exist zero divisors among locally slice functions.


\section{Factorization of zeros}\label{sec:factorization}

This section shows how, if $\Omega$ is a slice domain, the zeros of a slice regular function $f:\Omega\to\hh$ can be factored out, although possibly not on the whole domain $\Omega$. We begin with the case of non isolated zeros. For a fixed $2$-sphere $S=x_0+y_0\s$ (with $x_0,y_0\in\rr$ and $y_0>0$) that intersects $\Omega$, we preliminarily set
\[\mathcal{O}(S):=\{q\in S\cap\Omega:f^\circ_s(q)\neq0\mathrm{\ or\ }f'_s(q)\neq0\}\]
and notice that $(S\cap\Omega)\setminus\mathcal{O}(S)$ is the union of the connected components $C$ of $S\cap\Omega$ such that $f_{|_C}\equiv0$.

\begin{theorem}\label{thm:sphericalfactorization}
Let $\Omega$ be a slice domain in $\hh$ and let $f:\Omega\to\hh$ be a slice regular function. Fix a $2$-sphere $S=x_0+y_0\s$ (with $x_0,y_0\in\rr$ and $y_0>0$) that intersects $\Omega$ and consider the slice domain $\Omega':=\Omega\setminus\mathcal{O}(S)$. Then, there exists a slice regular function $h:\Omega'\to\hh$ such that
\[f(q)=[(q-x_0)^2+y_0^2]*h(q)=[(q-x_0)^2+y_0^2]h(q)\]
throughout $\Omega'$. Moreover, $h$ cannot be continuously extended to any open set intersecting $\mathcal{O}(S)$.
\end{theorem}

\begin{proof}
Let us first define a slice regular $h:\Omega\setminus S\to\hh$ by setting 
\[h(q):=[(q-x_0)^2+y_0^2]^{-*}*f(q)=[(q-x_0)^2+y_0^2]^{-1}f(q)\,.\]
If $\mathcal{O}(S)=S\cap\Omega$, then $\Omega'=\Omega\setminus S$ and the thesis immediately follows. Let us now suppose $\mathcal{O}(S)$ to be a proper subset of $S\cap\Omega$, whence $S\cap\Omega'\neq\emptyset$. Each connected component $C$ of $S\cap\Omega$ either does not intersect $S\cap\Omega'$ or is entirely included in $S\cap\Omega'$. The latter happens if, and only if, $(f^\circ_s)_{|_C}\equiv0\equiv(f'_s)_{|_C}$, which is in turn equivalent to $f_{|_C}\equiv0$.

Pick $\widetilde p\in S\cap\Omega'$ and note that $f_{|_C}\equiv0$ in the connected component $C$ of $S\cap\Omega$ that includes $\widetilde p$. We will extend both the slice regular function $h$ and the validity of the equality $f(q)=[(q-x_0)^2+y_0^2]*h(q)$ to a slice domain that includes the point $\widetilde p$.

By Theorem~\ref{localextension}, $f$ has an extension triplet $\big(\widetilde f,N,\Lambda\big)$ with $\widetilde p\in \Lambda$. Since $(\widetilde f)_{|_{C\cap\Lambda}}\equiv0$, we can apply~\cite[Theorem 3.1]{librospringer} to conclude that $(\widetilde f)_{|_S}\equiv0$.

By~\cite[Proposition 3.17]{librospringer}, there exists a slice regular function $\widetilde h : N \to \hh$ such that $\widetilde f(q)=[(q-x_0)^2+y_0^2]*\widetilde h(q)$ in $N$. As a consequence,
\[f(q)=[(q-x_0)^2+y_0^2]*\widetilde h(q)\]
in $\Lambda$. Now, $\widetilde h$ coincides with $h$ in $\Lambda\setminus S$. This completes the proof of the first statement.

We are left with proving that the function $h$ cannot admit a continuous extension to any open subset that intersects $\mathcal{O}(S)$ in a non empty subset $Y$. If it did then, by the equality $f(q)=[(q-x_0)^2+y_0^2]h(q)$ valid in $\Omega'$, $f$ would vanish identically in the intersection $Y$. This would contradict the definition of $\mathcal{O}(S)\supseteq Y$.
\end{proof}

We are now ready to prove that it is also possible to factor out single zeros. For a fixed $2$-sphere $S=x_0+y_0\s$ (with $x_0,y_0\in\rr$ and $y_0>0$) that intersects $\Omega$, we preliminarily set
\[\mathcal{O}(p):=\{q\in S\cap\Omega : f^\circ_s(q)\neq-\im(p)f'_s(q)\}\,.\]
Notice that $p\in(S\cap\Omega)\setminus\mathcal{O}(p)$ if, and only if, $f(p)=0$.

\begin{theorem}\label{thm:factorization}
Let $\Omega$ be a slice domain in $\hh$ and let $f:\Omega\to\hh$ be a slice regular function. Fix $p\in\Omega$ and consider the slice domain $\Omega':=\Omega\setminus\mathcal{O}(p)$. Then there exists a slice regular function $g:\Omega'\to\hh$ such that
\[f(q)=(q-p)*g(q)\]
throughout $\Omega'$. Moreover, for any slice domain $\Lambda$ that intersects $\mathcal{O}(p)$, there cannot exist a slice regular $g:\Lambda\to\hh$ such that the previous equality holds in $\Lambda$.
\end{theorem}

\begin{proof}
If $\mathcal{O}(p)=S\cap\Omega$, then $\Omega'=\Omega\setminus S$ and first statement is easily proven by setting $g(q)=(q-p)^{-*}*f(q)$. Let us now suppose $\mathcal{O}(p)$ to be a proper subset of $S\cap\Omega$, whence $S\cap\Omega'\neq\emptyset$. Each connected component $C$ of $S\cap\Omega$ either does not intersect $S\cap\Omega'$ or is entirely included in $S\cap\Omega'$. The latter happens if, and only if, $f^\circ_s(q)=-\im(p)f'_s(q)$ for all $q\in C$. Moreover, if $C$ includes $p$, then the latter equality is equivalent to $f(p)=0$.

Setting $\ell(q):=(q-\bar p)*f(q)$ defines a slice regular function $\ell:\Omega\to\hh$. Moreover, in each $2$-sphere $x+y\s$ it holds:
\begin{align*}
&\ell^\circ_s(q)=(x-\bar p)f^\circ_s(q)-y^2f'_s(q)\,,\\
&\ell'_s(q)=(x-\bar p)f'_s(q)+f^\circ_s(q)\,.
\end{align*}
In particular, for every connected component $C$ of $S\cap\Omega'$ and for all $q\in C$, it holds:
\begin{align*}
&\ell^\circ_s(q)=\im(p)f^\circ_s(q)-|\im(p)|^2f'_s(q)\,,\\
&\ell'_s(q)=\im(p)f'_s(q)+f^\circ_s(q)\,,\\
&f^\circ_s(q)=-\im(p)f'_s(q)\,.
\end{align*}
We conclude that $(\ell^\circ_s)_{|_C}\equiv0\equiv(\ell'_s)_{|_C}$, whence $\ell_{|_C}\equiv0$. Thus, $\ell_{|_{S\cap\Omega'}}\equiv0$. By an appropriate application of Theorem~\ref{thm:sphericalfactorization}, we conclude that
\[\ell(q)=[(q-x_0)^2+y_0^2]*g(q)=(q-\bar p)*(q-p)*g(q)\]
for some slice regular $g:\Omega'\to\hh$. Remembering that $\ell(q)=(q-\bar p)*f(q)$ throughout $\Omega$, we conclude that $f(q)=(q-p)*g(q)$ in $\Omega'$, as desired.

Finally, if $\Lambda$ is a slice domain where $f(q)=(q-p)*g(q)$ for some slice regular $g:\Lambda\to\hh$, then at each $q\in S\cap\Lambda$ it holds
\begin{align*}
f^\circ_s(q)&=-\im(p)\,g^\circ_s(q)+\im(p)^2g'_s(q)\,,\\
f'_s(q)&=-\im(p)\,g'_s(q)+g^\circ_s(q)\,,
\end{align*}
whence $f^\circ_s(q)=-\im(p)f'_s(q)$. The latter equality is false if $q\in\mathcal{O}(p)$.
\end{proof}

The next result addresses the problem of factoring out real zeros.

\begin{proposition}\label{prop:realfactorization}
Let $\Omega$ be a slice domain in $\hh$, let $f:\Omega\to\hh$ be a slice regular function and fix $x\in\Omega\cap\rr$. It holds $f(x)=0$ if, and only if, there exists a slice regular function $g:\Omega\to\hh$ such that
\[f(q)=(q-x)*g(q)=(q-x)g(q)\]
throughout $\Omega$.
\end{proposition}

\begin{proof}
Let us first define a slice regular $g:\Omega\setminus\{x\}\to\hh$ by setting 
\[g(q):=(q-x)^{-*}*f(q)=(q-x)^{-1}f(q)\,.\]
If $f(x)\neq0$, then $g$ cannot be continuously extended to the point $x$. If, instead, $f(x)=0$, let us consider an open Euclidean ball $B$ centered at $x$ and included in $\Omega$: by~\cite[Proposition 3.17]{librospringer}, $g_{|_{B\setminus\{x\}}}$ extends to a slice regular function $\widetilde g:B\to\hh$ fulfilling the equality $f(q)=(q-x)*\widetilde g(q)$ in $B$.
\end{proof}

\begin{definition}
Let $\Omega$ be a slice domain in $\hh$ and let $f:\Omega\to\hh$ be a slice regular function. If there exist slice regular functions $h,g:\Omega\to\hh$ such that
\[f(q)=h(q)*g(q)\]
in $\Omega$, we say that $h(q)$ \emph{(left) divides} $f(q)$. Now suppose $C$ to be a connected component of $(x+y\s)\cap\Omega$: we say that $h(q)$ \emph{(left) divides} $f(q)$ \emph{near} $C$ if there exist a slice domain $\Lambda$ with $C\subseteq\Lambda\subseteq\Omega$ and slice regular functions $h,g:\Lambda\to\hh$ such that
\[f(q)=h(q)*g(q)\]
in $\Lambda$. Finally, for any point $p_0\in x+y\s$, we say that $h(q)$ \emph{(left) divides} $f(q)$ \emph{near} $p_0$ if $h(q)$ left divides $f(q)$ near the connected component of $(x+y\s)\cap\Omega$ that includes $p_0$.
\end{definition}

We now present examples of slice regular functions $g:\Omega\to\hh$ and points $\widetilde p=x+yJ\in\Omega$ such that $q-\widetilde p$ divides $g(q)$
\begin{enumerate}
\item near the connected component of $(x+y\s)\cap\Omega$ that includes $\widetilde p$ (whence $g(\widetilde p)=0$);
\item near a connected component of $(x+y\s)\cap\Omega$ that does not include $\widetilde p$ (while $g$ has no zeros in $x+y\s$);
\item globally in $\Omega$, despite the fact the $g^\circ_s,g'_s$ are not constant in $(x+y\s)\cap\Omega$ (a case when $g(\widetilde p)=0$, too).
\end{enumerate}

\begin{example}\label{ex:factorization}
Let $f:\Omega\to\hh$ be the slice regular function of Examples~\ref{ex:douren} and~\ref{ex:douren2}. Recall that we have set $p:=-1+2I$. Let us pick $J\in\s$ and set $\widetilde p:=-1+2J$. We define
\[v:=f^\circ_s(p)+\im(\widetilde p)f'_s(p)=\frac12\left[\phi_0(\bar p)-\pi I\right]+\frac{J}2\left[I\phi_0(\bar p)-\pi\right]\]
and notice that
\begin{enumerate}
\item If $\widetilde p\in C^+$, then $v=f^\circ_s(\widetilde p)+\im(\widetilde p)f'_s(\widetilde p)=f(\widetilde p)$.
\item If $\widetilde p\in C^-\setminus\{\bar p\}$, then $v\neq f(\widetilde p)$ because $f(\widetilde p)-v=\frac12\left[\phi_0(\bar p)+\pi I\right]+\frac{J}2\left[I\phi_0(\bar p)+\pi\right]-v=\pi(I+J)\neq0$.
\item If $\widetilde p=\bar p$, then $v=f(\bar p)$ because $f(\widetilde p)-v=\pi(I-I)=0$.
\end{enumerate}
We define
\[g:\Omega\to\hh,\quad q\mapsto f(q)-v\,.\]
By direct computation,
\begin{align*}
(g^\circ_s)_{|_{C^+}}&\equiv-\frac{J}2\left[I\phi_0(\bar p)-\pi\right]\,,\\
(g'_s)_{|_{C^+}}&\equiv\frac14\left[I\phi_0(\bar p)-\pi\right]\,,
\end{align*}
whence the constant values of $(g^\circ_s)_{|_{C^+}}$ and $-2J(g'_s)_{|_{C^+}}=-\im(\widetilde p)(g'_s)_{|_{C^+}}$ coincide. On the other hand,
\begin{align*}
(g^\circ_s)_{|_{C^-}}&\equiv-\frac{JI}2\phi_0(\bar p)+\pi\left(I+\frac{J}2\right)\,,\\
(g'_s)_{|_{C^-}}&\equiv\frac14\left[I\phi_0(\bar p)+\pi\right]\,.
\end{align*}
If $J=-I$, then the constant values of $(g^\circ_s)_{|_{C^-}}$ and $-\im(\widetilde p)(g'_s)_{|_{C^-}}$ coincide. If, instead, $J\neq-I$, we can prove that, for all $\widehat p\in(-1+2\s)\cap\Omega$, the constant values of $(g^\circ_s)_{|_{C^-}}$ and $-\im(\widehat p)(g'_s)_{|_{C^-}}$ are distinct. Indeed, the constant value of
\[\left|(g^\circ_s)_{|_{C^-}}\right|^2-4\left|(g'_s)_{|_{C^-}}\right|^2=\left|J(g^\circ_s)_{|_{C^-}}\right|^2-4\left|(g'_s)_{|_{C^-}}\right|^2\]
is different from $0$, according to the following computation:
\begin{align*}
&\left|\frac{I}2\phi_0(\bar p)+\pi\left(JI-\frac12\right)\right|^2-\left|\frac{I}2\phi_0(\bar p)+\frac\pi2\right|^2\\
&=\pi^2\left|JI-\frac12\right|^2+\pi\re\left(\left(IJ-\frac12\right)I\phi_0(\bar p)\right)-\frac{\pi^2}4-\pi\re\left(\frac{I}2\phi_0(\bar p)\right)\\
&=\pi^2(1-\re(JI))+\pi\re\left(\left(IJ-1\right)I\phi_0(\bar p)\right)\\
&=\pi(1-\re(JI))(\pi+\arg_0(-1-4I))\neq0\,.
\end{align*}
The last equality holds because $I\phi_0(\bar p)=-\arg_0(-1-4I)+I\ln\sqrt{17}\in L_I$. There are three possibilities:
\begin{enumerate}
\item If $\widetilde p\in C^+$, then $q-\widetilde p$ divides $g(q)$ near $C^+$ but not near $C^-$. Moreover,
\[Z(g)\cap C^+=\{\widetilde p\},\quad Z(g)\cap C^-=\emptyset\]
\item If $\widetilde p\in C^-\setminus\{\bar p\}$, then $q-\widetilde p$ divides $g(q)$ near $C^+$ but not near $C^-$. Moreover,
\[Z(g)\cap C^+=\emptyset,\quad Z(g)\cap C^-=\emptyset\]
\item If $\widetilde p=\bar p$, then $q-\bar p$ divides $g(q)$ globally in $\Omega$. Moreover,
\[Z(g)\cap C^+=\emptyset,\quad Z(g)\cap C^-=\{\bar p\}\,.\]
\end{enumerate}
\end{example}

We conclude this section with the next definition, consistent with~\cite[Definition 3.26]{librospringer} (which, in turn, derived from~\cite{zeros}).

\begin{definition}
Let $\Omega$ be a slice domain in $\hh$ and let $f:\Omega\to\hh$ be a slice regular function with $f\not\equiv0$. Fix an $S:=x+y\s$ (with $x,y\in\rr$) intersecting $\Omega$, a connected component $C$ of the intersection $S\cap\Omega$ and a point $p\in S$. We define the \emph{classical multiplicity of $f$ at $p$ relative to $C$}, and denote by $m_f^C(p)$, the number $n\in\nn$ such that $(q-p)^{*n}$ divides $f(q)$ near $C$ but $(q-p)^{*(n+1)}$ does not. If, moreover, $S\cap\Omega$ has a unique connected component $C$, then we also call $m_f^C(p)$ the \emph{classical multiplicity of $f$ at $p$} and denote it as $m_f(p)$.
\end{definition}

If $x\in\Omega\cap\rr$, then of course $x+0\s=\{x\}$ intersects $\Omega$ in a connected set: namely, the singleton $\{x\}$. Thus, it is automatically possible to write $m_f(x)$ without specifying a connected component $C$.

The previous definition is well posed for the following reason. Supposing $p=x+yI$, any equality $f(q)=(q-p)^{*n}*g(q)$ valid in a slice domain $\Lambda$ implies an equality $f_I(z)=(z-p)^ng_I(z)$ valid for all $z\in\Lambda_I$. Moreover, $f_I$ is a holomorphic function: if for all $n\in\nn$ there exist holomorphic functions $\phi^{[n]}$ such that $f_I(z)=(z-p)^n\phi^{[n]}(z)$, then $f_I\equiv0$. But in such a case, it holds $f\equiv0$ by the Identity Principle~\ref{identity}.

\begin{example}
Let us consider again the functions $g$ and the points $\widetilde p$ appearing in Example~\ref{ex:factorization}, in the three separate cases listed there.
\begin{enumerate}
\item It holds $g(\widetilde p)=0$ and $m_{g}^{C^+}(\widetilde p)\geq1$; on the other hand, $m_{g}^{C^-}(\widetilde p)=0$.
\item It holds $g(\widetilde p)\neq0$ and $m_{g}^{C^+}(\widetilde p)\geq1$; on the other hand, $m_{g}^{C^-}(\widetilde p)=0$.
\item It holds $g(\widetilde p)=0$ and $m_{g}^{C^+}(\widetilde p)\geq1,\,m_{g}^{C^-}(\widetilde p)\geq1$.
\end{enumerate}
\end{example}


\section{Applications of factorization}\label{sec:applications}

The results of the previous section allow us to provide explicit examples of all types of zeros envisioned in Section~\ref{sec:zerosets}. In the first example, the given $g$ admits exactly one zero in the given $2$-sphere $x+y\s$:

\begin{example}\label{ex:1point}
Let us refer to Example~\ref{ex:factorization} and pick $J=I$, whence $\widetilde p=p\in C^+$. Then, $g(q):=f(q)-f(p)=f(q)+\pi I$ has
\[Z(g)\cap C^+=\{p\},\quad Z(g)\cap C^-=\emptyset\]
Moreover, $q-p$ divides $g(q)$ near $C^+$ but not near $C^-$.
\end{example}

In the second example, the given $2$-sphere $x+y\s$ intersects the domain in two connected components. The function $\ell$ considered vanishes identically on one component and admits exactly one zero in the other component:

\begin{example}\label{ex:1cap1point}
Consider again the function $g:\Omega\to\hh$ in Example~\ref{ex:1point}. Let us define $\ell:\Omega\to\hh$ by setting
\[\ell(q):=(q-\bar p)*g(q)\,.\]
Then $(q-p)^s=(q-\bar p)*(q-p)=q^2+2q+5$ divides $\ell(q)$ near $C^+$ but not near $C^-$. Moreover, $q-\bar p$ divides $\ell(q)$ globally in $\Omega$. As a consequence,
\[Z(\ell)\supseteq C^+,\quad Z(\ell)\cap C^-=\{\bar p\}\,.\]
\end{example}

In the third example, too, the given $2$-sphere $x+y\s$ intersects the domain in two connected components. The function $m$ considered admits exactly one zero in each connected component:

\begin{example}\label{ex:2points}
Consider again the function $g:\Omega\to\hh$ in Example~\ref{ex:1point}. The point $p=-1+2I$ was a zero of $g$ and the binomial $q-p$ divided $g(q)$ near $C^+$. 

Let us choose $I_0\in\s\setminus\{-I\}$ with $|I_0-I|>\frac12$. Let us consider the points
\begin{align*}
p_0&:=-1+2I_0 \in C^-\,,\\
p_1&:=g(p_0)^{-1}p_0g(p_0)=-1+2I_1\,,
\end{align*}
where $I_1:=g(p_0)^{-1}I_0g(p_0)\in\s$. For future reference, we prove that $I_1\neq\pm I$. Indeed, the following equalities are equivalent:
\begin{align*}
g(p_0)^{-1}I_0g(p_0)&=\pm I\,,\\
I_0g(p_0)&=g(p_0)(\pm I)\,,\\
I_0g^\circ_s(p_0)-2g'_s(p_0)&=\pm Ig^\circ_s(p_0)\pm 2I_0Ig'_s(p_0)\,,\\
(I_0\mp I)g^\circ_s(p_0)&=2(1\pm I_0I)g'_s(p_0)\,,\\
g^\circ_s(p_0)&=\pm2Ig'_s(p_0)\,;
\end{align*}
the last equality being false, because we proved in Example~\ref{ex:factorization} that $|g^\circ_s(p_0)|\neq2|g'_s(p_0)|$.

Let us define $m:\Omega\to\hh$ as the function
\[m(q):=g(q)*(q-p_1)=qg(q)-g(q)p_1\,.\]
Since $q-p$ divides $g(q)$ near $C^+$, it also divides $m(q)$ near $C^+\ni p$. Thus, $m(p)=0$. 
Now,
\[m'_s(p)=g^\circ_s(p)-g'_s(p)\im(p_1)=-g'_s(p)(\im(p)+\im(p_1))=-2g'_s(p)(I+I_1)\,,\]
 is different from $0$ becuase $g'_s(p)\neq0$ and $I_1\neq-I$. Thus,
\[Z(m)\cap C^+=\{p\}\,.\]

Let us turn to the connected component $C^-$ of $(-1+2\s)\cap\Omega$, which includes $\bar p$ and $p_0$. By the definition of $p_1$, it holds
\[m(p_0)=p_0g(p_0)-g(p_0)p_1=p_0g(p_0)-p_0g(p_0)=0\,.\]
Now,
\[m'_s(p_0)=g^\circ_s(p_0)-g'_s(p_0)\im(p_1)=g^\circ_s(p_0)-2g'_s(p_0)I_1\]
cannot vanish, because we know that $|g^\circ_s(p_0)|\neq2|g'_s(p_0)|$. Thus,
\[Z(m)\cap C^-=\{p_0\}\,.\]
\end{example}

In the previous Example~\ref{ex:2points} we have excluded $I_0=-I$ because, in such a case, we would have had $p_0=\bar p$, $p_1=\bar p$ and a function $m(q)=g(q)*(q-\bar p)=(q-\bar p)*g(q)$ coinciding with the function $\ell(q)$ of the preceding Example~\ref{ex:1cap1point}.

As a further application of the results of the previous section, we can complete the characterizations of $Z(f*g)$ and $Z(f^s)$ initiated in Propositions~\ref{prop:zerosofproduct} and~\ref{prop:zerosofsymmetrization1}. Recall that, for $p$ and $\widetilde p$ belonging to the same $x+y\s$, we say that $q-\widetilde p$ divides $f(q)$ near $p$ if the following property holds: there exist a slice domain $\Lambda$ containing the connected component $C$ of $(x+y\s)\cap\Omega$ that includes $p$ and a slice regular function $h:\Lambda\to\hh$ such that
\[f(q)=(q-\widetilde p)*h(q)\]
in $\Lambda$. In such a case, we also say that $q-\widetilde p$ divides $f(q)$ near $C$.

\begin{proposition}\label{prop:zerosofproductandsymmetrization}
Let $\Omega$ be a slice domain in $\hh$ and let $f,g:\Omega\to\hh$ be slice regular functions. Fix an $S:=x+y\s$ intersecting $\Omega$ and a connected component $C$ of the intersection $S\cap\Omega$. 
\begin{enumerate}
\item $Z(f*g)$ is the union between $Z(f)$ and the set of points $p\in\Omega\setminus Z(f)$ such that $q-f(p)^{-1}pf(p)$ divides $g(q)$ near $p$.
\item $Z(f^s)$ intersects $C$ if, and only if, it includes it; this happens if, and only if, there exists $\widetilde p\in S$ such that $q-\widetilde p$ divides $f(q)$ near $C$.
\item There exists $\widetilde p\in S$ such that $q-\widetilde p$ divides $f(q)$ near $C$ if, and only if, there exists $\widehat p\in S$ such that $q-\widehat p$ divides $f^c(q)$ near $C$.
\end{enumerate}
\end{proposition}

\begin{proof}
\begin{enumerate}
\item The thesis immediately follows from Proposition~\ref{prop:zerosofproduct} and Theorem~\ref{thm:factorization}.
\item If there exist $\widetilde p\in S$ such that $q-\widetilde p$ divides $f(q)$ near $C$, then $(q-\widetilde p)^s=(q-x)^2+y^2$ divides $f^s(q)$ near $C$, whence $f^s$ vanishes identically in $C$.\\
Conversely, let us suppose $f^s_{|_C}\equiv0$. If $y=0$, then $x+y\s=\{x\}=C$ and $f^s(x)=0$ implies $f(x)=0$. In this situation, Proposition~\ref{prop:realfactorization} guarantees that $q-x$ divides $f(q)$ globally in $\Omega$. Now suppose, instead, $y\neq0$. Then $(f^s)^\circ_s$ and $(f^s)'_s$ vanish identically in $C$. Let $b,c$ denote the constant values of $f^\circ_s$ and $yf'_s$ in $C$. By Definition~\ref{def:operations}, it follows that $|b|^2=|c|^2$ and $\re(b\bar c)=0$. Thus, there exists $I\in\s$ such that $b\bar c=I|c|^2$, whence $b=Ic$. If we set $\widetilde p=x-yI$, then $f^\circ_s(q)=-\im(\widetilde p)f'_s(q)$ for all $q\in C$. We can now apply Theorem~\ref{thm:factorization} to conclude that $q-\widetilde p$ divides $f(q)$ near $C$. This concludes our proof.
\item The thesis follows from property {\it 2.}, taking into account the equality $(f^c)^s=f^c*(f^c)^c=f^c*f=f^s$.\qedhere
\end{enumerate}
\end{proof}

We now show, with an example, that ``phantom'' zeros of $g$ may produce zeros of $f*g$ and of $g^s$.

\begin{example}
Let us refer to case {2.} in Example~\ref{ex:factorization}: $\widetilde p\in C^-\setminus\{\bar p\}$ and $q-\widetilde p$ divides $g(q)$ near $C^+$, but not near $C^-$. Moreover,
\[Z(g)\cap C^+=\emptyset,\quad Z(g)\cap C^-=\emptyset\,.\]
Let us consider the binomial $B(q)=q-\overline{\widetilde p}$, which has
\[Z(B)=\left\{\overline{\widetilde p}\right\}\,,\]
and the product $B*g(q)$. Then $(q-\widetilde p)^s=q^2+2q+5$ divides $B*g(q)$ near $C^+$. Thus,
\[Z(B*g)\supseteq C^+\,.\]
Let us now consider the symmetrization $g^s$: since  $q-\widetilde p$ divides $g(q)$ near $C^+$, it follows that
\[Z(g^s)\supseteq C^+\,,\]
despite the fact that $Z(g)\cap C^+=\emptyset$.
Finally, we observe that there exists $\widehat p\in -1+2\s$ such that $q-\widehat p$ divides $f^c(q)$ near $C^+$.
\end{example}

We now complete the terminology concerning multiplicities with the next result and definition. Notice that, for $p=x+yI,p'\in\hh$, it holds $(q-p)*(q-p')=(q-x)^2+y^2$ if, and only if, $p'=\bar p$.

\begin{theorem}\label{factorization}
Let $\Omega$ be a slice domain in $\hh$ and let $f:\Omega\to\hh$ be a slice regular function with $f\not\equiv0$. Fix an $S:=x+y\s$ (with $x,y\in\rr,y>0$) intersecting $\Omega$ and a connected component $C$ of the intersection $S\cap\Omega$. There exist $m \in \nn$ such that $[(q-x)^2+y^2]^m$ divides $f(q)$ near $C$ but $[(q-x)^2+y^2]^{m+1}$ does not. Moreover, for the function $h$ such that
\[f(q) = \big[(q-x)^2+y^2\big]^m h(q)\]
near $C$, there exist a number $n \in \nn$, points $p_1,\ldots,p_n\in S$ (with $p_i \neq\bar p_{i+1}$ for all $i \in \{1,\ldots,n-1\}$), and a slice regular function $g$ (with $g^s_{|_C}\not\equiv0$) such that
\[h(q) = (q-p_1)*(q-p_2)*\cdots*(q-p_n)*g(q)\]
near $C$.
\end{theorem}

\begin{proof} 
Since $f \not \equiv 0$ in $\Omega$, by Theorem~\ref{thm:sphericalfactorization} there exists an $m \in \nn$ such that
\[f(q) = [(q-x)^2+y^2]^m h(q)\]
near $C$, for some slice regular $h$ which does not vanish identically in $C$. Suppose, indeed, it were possible to find, for all $k \in \nn$, a function $h^{[k]}(q)$ such that $f(q) = [(q-x)^2+y^2]^k h^{[k]}(q)$ in $C$. Then, choosing a point $x+yI\in C$, the restriction $f_I$ would have the factorization 
\[f_I(z) = [(z-x)^2+y^2]^k h^{[k]}_I(z) = [z-(x+yI)]^k[z-(x-yI)]^k h^{[k]}_I(z)\]
for all $k \in \nn$. This would imply $f_I \equiv 0$, whence $f \equiv 0$ by the Identity Principle \ref{identity}.

Now consider any slice regular function $h$ on a slice domain $\Lambda$ including $C$, with $h_{|_C}\not\equiv0$. Let us apply Proposition~\ref{prop:zerosofproductandsymmetrization} to $g^{[0]}:=h$: if the symmetrization of $g^{[0]}$ vanishes in $C$, then there exists $p_1\in S$ such that $h(q) = (q-p_1) * g^{[1]}(q)$ near $C$, for some slice regular $g^{[1]}$ with $g^{[1]}_{|_C}\not\equiv0$. If for all $k\in\nn$ there existed a $p_{k+1}\in S$ and a $g^{[k+1]}$ such that $g^{[k]}(q) = (q-p_{k+1})*g^{[k+1]}$ near $C$, then we would have
\[h(q)=(q-p_1)*\ldots*(q-p_k)*g^{[k]}(q)\]
near $C$ for all $k \in \nn$. This would imply, for the symmetrization $h^s$ of $h$, 
\[h^s(q)=[(q-x)^2+y^2]^k(g^{[k]})^s(q)\]
near $C$ for all $k\in\nn$. By the first part of the proof, this would imply $h^s\equiv0$. But then Proposition \ref{prop:zerosofsymmetrization1} would yield $h\equiv0$, a contradiction. Thus there exists an $n\in\nn$ such that the symmetrization of $g^{[n]}$ does not vanish in $C$. The thesis follows, setting $g:=g^{[n]}$.
\end{proof}

We notice that, in the previous statement, the numbers $m,n$ are uniquely determined. The points $p_1,\ldots,p_n$ are also uniquely determined, by~\cite[Proposition 3.23]{librospringer}. We are now in a position to give the next definition.

\begin{definition}\label{R-molt.perregolari}\index{multiplicity of a zero!spherical}\index{multiplicity of a zero!isolated}
Let $\Omega$ be a slice domain in $\hh$ and let $f:\Omega\to\hh$ be a slice regular function.

In the situation described in Theorem~\ref{factorization}, the number $2m$ is called the \emph{spherical multiplicity of $f$ at $C$} and the number $n$ is called the \emph{isolated multiplicity of $f$ at $p_1$, relative to $C$}. If, moreover, $S\cap\Omega$ has a unique connected component, then we also call $2m$ the \emph{spherical multiplicity of $f$ at $S\cap\Omega$} and call $n$ the \emph{isolated multiplicity of $f$ at $p_1$}.

For $x\in\Omega\cap\rr$, we define the \emph{isolated multiplicity of $f$ at $x$} to coincide with the classical multiplicity $m_f(x)$.
\end{definition}

The previous definition is consistent with~\cite[Definition 3.37]{librospringer}, which in turn derived from~\cite{zerosopen,milan}.

We conclude this section using factorization to characterize the set of points $p$ such that the real differential $df_p$ of a slice regular function $f$ at $p$ is singular. The corresponding result for symmetric slice domains is~\cite[Theorem 8.20]{librospringer} (from~\cite{ocs}). We follow here a different approach, analogous to that of~\cite{perottipersonal}.

\begin{theorem}
Let $\Omega$ be a slice domain in $\hh$ and let $f:\Omega\to\hh$ be a slice regular function. Fix $p\in\Omega$ and set $h:=f-f(p)$.
\begin{enumerate}
\item If $p\in\rr$, then $df_p$ is singular if, and only if, the Cullen derivative $f'_c(p)$ vanishes. This happens if, and only if, $m_h(p)\geq2$.
\item Suppose, instead, $p=x+yI$ with $x,y\in\rr$ and $y>0$ and let $C$ denote the connected component of $(x+y\s)\cap\Omega$ including $p$. Let us denote by $2m$ the spherical multiplicity of $h$ at $C$ and by $n$ the isolated multiplicity of $h$ at $p$, relative to $C$. Then $df_p$ is singular if, and only if, $2m+n\geq2$. In particular, the spherical derivative $f'_s(p)$ vanishes if, and only if, $2m\geq2$, while the Cullen derivative $f'_c(p)$ vanishes if, and only if, $m_h^C(p)\geq2$.
\end{enumerate}
\end{theorem}

\begin{proof}
We remark that $h(p)=0$.
\begin{enumerate}
\item If $p\in\rr$, Proposition~\ref{prop:realfactorization} guarantees the existence of a slice regular $g:\Omega\to\hh$ such that
\[h(q)=(q-p)*g(q)\,.\]
Since $f$ and $h$ differ by an additive constant, it holds
\[f'_c(q)=h'_c(q)=g(q)+(q-p)*g'_c(q)\,,\]
whence $f'_c(p)=g(p)$. By Proposition~\ref{prop:differential}, $df_p$ is singular if, and only if, $f'_c(p)=g(p)$ equals zero. We can apply Proposition~\ref{prop:realfactorization} again to see that $g(p)=0$ if, and only if, there exists a slice regular function $\ell:\Omega\to\hh$ such that
\[g(q)=(q-p)*\ell(q)\,,\]
which is, in turn, equivalent to
\[h(q)=(q-p)^{*2}*\ell(q)\,.\]
\item Suppose we are in the opposite case. Theorem~\ref{thm:factorization} guarantees the existence of a slice domain $M$ containing $C\cup(\Omega\setminus(x+y\s))$ and of a slice regular function $g:M\to\hh$ such that
\[h(q)=(q-p)*g(q)\]
in $M$. For all $q\in M$, it holds
\begin{align*}
&f'_c(q)=h'_c(q)=g(q)+(q-p)*g'_c(q)\\
&f'_s(q)=h'_s(q)=g^\circ_s(q)+(\re(q)-p)g'_s(q)\,.
\end{align*}
In particular,
\begin{align*}
&f'_c(p)=g(p)\\
&f'_s(p)=g^\circ_s(p)-\im(p)g'_s(p)\,,
\end{align*}
whence
\begin{align*}
&f'_c(p)\overline{f'_s(p)}=g(p)\overline{(g^\circ_s(p)-\im(p)g'_s(p))}\\
&=\big(g^\circ_s(p)+\im(p)g'_s(p)\big)\big(\overline{g^\circ_s(p)}+\overline{g'_s(p)}\im(p)\big)\\
&=\big|g^\circ_s(p)\big|^2-\big|\im(p)\big|^2\big|g'_s(p)\big|^2+2\im\big(g^\circ_s(p)\overline{g'_s(p)}\im(p)\big)\\
&=(g^s)^\circ_s(p)+2\re\big(g^\circ_s(p)\overline{g'_s(p)}\big)\im(p)+2\im\big(g^\circ_s(p)\overline{g'_s(p)}\big)\times\im(p)\\
&=(g^s)^\circ_s(p)+\im(p)(g^s)'_s(p)+2\im\big(g^\circ_s(p)\overline{g'_s(p)}\big)\times\im(p)\\
&=\underbrace{g^s\big(p\big)}_{\in L_I}+\underbrace{2\im\big(g^\circ_s(p)\overline{g'_s(p)}\big)\times\im(p)}_{\in L_I^\perp}\,.
\end{align*}
\end{enumerate}
By Proposition~\ref{prop:differential}, $df_p$ is singular if, and only if, $f'_c(p)\overline{f'_s(p)}\in L_I^\perp$. This is, in turn, equivalent to $g^s(p)=0$ by the previous formula. By Proposition~\ref{prop:zerosofproductandsymmetrization}, $g^s(p)=0$ if, and only if, there exist $\widetilde p\in x+y\s$, a slice domain $\Omega'\subseteq M$ containing both $C$ and $\Omega\setminus(x+y\s)$, as well as a slice regular function $\ell:\Omega'\to\hh$ such that
\[g(q)=(q-\widetilde p)*\ell(q)\]
in $\Omega'$. This equality, in turn, is equivalent to
\[h(q)=(q-p)*(q-\widetilde p)*\ell(q)\,.\]
in $\Omega'$. Finally, we remark that $f'_c(p)=0$ if, and only if, $\widetilde p$ can be chosen to equal $p$, while $f'_s(p)=0$ if, and only if, $\widetilde p$ can be chosen to equal $\bar p$.
\end{proof}


\section{Singularities}\label{sec:singularities}

This section is devoted to classifying the singularities of slice regular functions on slice domains. The main tool is the possibility to expand into regular Laurent series, according to the next definition and results (from~\cite[\S 5.2]{librospringer}, which in turn derived from~\cite{poli,singularities}). Let the expressions $(q-p)^{*(-n)} = (q-p)^{-*n}$ denote the regular reciprocal of $(q-p)^{*n}$.

\begin{definition}\index{function $\tau$}
The functions $\sigma,\tau:\hh\times\hh\to\rr$ are defined by setting, for all $p,q \in \hh$,
\begin{align*}
\sigma(q,p) &:= \left\{
\begin{array}{l}
|q-p| \mathrm{\ if\ } p,q \mathrm{\ lie\ in\ the\ same\ plane\ } L_I\\
\omega(q,p) \mathrm{\ otherwise}
\end{array}
\right.\\
\tau(q,p) &:= \left\{
\begin{array}{l}
|q-p|  \mathrm{\ if\ } p,q \mathrm{\ lie\ in\ the\ same\ plane\ } L_I\\
\sqrt{\left[\re(q)-\re(p)\right]^2 + \left[|\im(q)| - |\im(p)|\right]^2}   \mathrm{\ otherwise}
\end{array}
\right.
\end{align*}
where
\[\omega(q,p) := \sqrt{\left[\re(q)-\re(p)\right]^2 + \left[|\im(q)| + |\im(p)|\right]^2}.\]
For $0\leq R_1 <R_2\leq +\infty$, the following sets are defined:
\begin{align*}
&\Sigma(p,R_2) := \{q \in \hh : \sigma(q,p) < R_2\}\\
&T(p,R_2) := \{q \in \hh : \tau(q,p) < R_2\}\\
&\Omega(p,R_2) := \{q \in \hh : \omega(q,p) <R_2\}\\
&\Sigma(p,R_1,R_2) := \{q \in \hh : \tau(q,p)>R_1, \sigma(q,p) <R_2\}\\
&\Omega(p,R_1,R_2) := \Omega(p,R_2)\setminus\overline{T(p,R_1)}\,,
\end{align*}
where $\overline{T(p,R_1)}$ denotes the closure of $T(p,R_1)$.
\end{definition}

\begin{theorem}
Choose any sequence $\{a_n\}_{n \in \zz}$ in $\hh$. Let $R_1,R_2 \in [0, +\infty]$ be such that $R_1 = \limsup_{m \to +\infty} |a_{-m}|^{1/m}, 1/R_2 = \limsup_{n \to +\infty} |a_n|^{1/n}$.
For all $p \in \hh$ the \emph{regular Laurent series centered at $p$} associated to $\{a_n\}_{n\in \zz}$, namely,
\[f(q) = \sum_{n \in \zz} (q-p)^{*n} a_n\,,\]
converges absolutely and uniformly on the compact subsets of $\Sigma(p,R_1,R_2)$ and it does not converge at any point of $T(p,R_1)$ nor at any point of $\hh \setminus \overline{\Sigma(p,R_2)}$. Furthermore: if $\Omega(p,R_1,R_2) \neq \emptyset$, then the sum of the series defines a slice regular function $f : \Omega(p,R_1,R_2) \to \hh$.
\end{theorem}

\begin{theorem}[Regular Laurent Expansion]\label{laurent}
Let $\Omega$ be a domain in $\hh$, let $f:\Omega\to\hh$ be a slice regular function and let $p \in \hh$. There exists a sequence $\{a_n\}_{n \in \zz}$ in $\hh$ such that, for all $0\leq R_1<R_2\leq +\infty$ with $\Sigma(p,R_1,R_2)\subseteq\Omega$,
\begin{equation}\label{laurentformula}
f(q) = \sum_{n \in \zz} (q-p)^{*n} a_n
\end{equation}
in $\Sigma(p,R_1,R_2)$. If, moreover, $\Sigma(p,R_2)\subseteq\Omega$, then $a_n=0$ for all $n<0$ and equality~\eqref{laurentformula} holds throughout $\Sigma(p,R_2)$. 
\end{theorem}

We now give some new definitions and results. For the case of symmetric slice domains, see~\cite[\S 5.3]{librospringer} (derived from~\cite{poli,singularities}).

\begin{definition}
Let $f:\Omega\to\hh$ be a slice regular function and fix $p \in \hh$. We say that $p$ is a \emph{singularity} for $f$ if there exists $R>0$ such that $\Sigma(p,0,R) \subseteq\Omega$, whence the Laurent expansion of $f$ at $p$, $f(q) = \sum_{n \in \zz} (q-p)^{*n}a_n$, has $0$ as its inner radius of convergence and has a positive outer radius of convergence. Suppose this to be the case: we say that $p$ is a \emph{pole} for $f$ if there exists an $m\geq0$ such that $a_{-k} = 0$ for all $k>m$; the minimum such $m$ is called the \emph{order} of the pole and denoted by $ord_f(p)$. If $p$ is not a pole, then we call it an \emph{essential singularity} for $f$ and set $ord_f(p) = +\infty$. Additionally, we call $p$ a \emph{removable singularity} if $f$ extends to a neighborhood of $p$ as a slice regular function.
\end{definition}

Every removable singularity has order $0$, but the converse implication is false. In the next example, a connected component of $(x+y\s)\cap\Omega$ consists of nonremovable singularities, including one having order $0$. 

\begin{example}
Let us consider again the function $g$ of Example~\ref{ex:1point}, which was divisible by $q-p$ in the connected component $C^+$ of $(-1+2\s)\cap\Omega$ but not in the complementary connected component $C^-$. Let us define a function $h:\Omega\setminus(-1+2\s)\to\hh$ by setting
\[h(q):=(q-p)^{-*}*g(q)=(q^2+2q+5)^{-1}(q-\bar p)*g(q)\,.\]
Then every $\widetilde p=-1+2J\in C^+$ is a removable singularity for $h$ and $ord_h(\widetilde p)=0$. On the other hand, every $\widetilde p=-1+2J\in C^-$ is a nonremovable singularity for $h$. Moreover, $ord_h(\widetilde p)\geq1$ for all $\widetilde p\in C^-$ except $\bar p$: indeed, the equality $h_I(z)=(z-p)^{-1}g_I(z)$, valid for all $z\in\Omega_I\setminus\{\bar p\}$, yields that $ord_h(\bar p)=0$.
\end{example}

\begin{definition}\label{def:semiregular}
Let $\Omega$ be a slice domain in $\hh$. A function $f$ is \emph{semiregular} in $\Omega$ if it is slice regular in a slice domain $\Omega'\subseteq\Omega$ such that every point of $\Omega\setminus\Omega'$ is a pole (or a removable singularity) for $f$ and such that, for any $x,y\in\rr$ and for each connected component $C$ of $(x+y\s)\cap\Omega$,
the supremum $\sup_{p\in C}ord_f(p)$ is finite.
\end{definition}

Notice that, if $f$ is semiregular in $\Omega$ and if $\mathcal{S}$ is the set of its nonremovable poles, then, for all $I \in \s$: the slice $\mathcal{S}_I$ is discrete; the restriction $f_I$ is holomorphic in $\Omega_I \setminus \mathcal{S}_I$; and $f_I$ is meromorphic in $\Omega_I$. We point out that $\Omega\setminus\mathcal{S}$ is a slice domain. We will see in Section~\ref{sec:sphericalseries} (specifically, in Remark~\ref{rmk:sphericalorder}) that the hypothesis $\sup_{p\in C}ord_f(p)<+\infty$ is not restrictive if $\mathcal{S}$ is a cap within a $2$-sphere of the form $x+y\s$. 

\begin{theorem}\label{polefactorization1}
Let $\Omega$ be a slice domain in $\hh$ and let $f$ be semiregular in $\Omega$. Fix $p = x+yI \in \Omega$ and let $C$ denote the connected component of $(x+y\s)\cap\Omega$ that includes $p$. Let $m:=ord_f(p)$ and $n:=\max_{\widetilde p\in C}ord_f(\widetilde p)$. There exist a slice domain $\Lambda$ with $C\subseteq\Lambda\subseteq\Omega$ and a slice regular function $g : \Lambda \to \hh$ such that
\[f(q) = \left[(q-x)^2+y^2\right]^{-n} (q-p)^{*(n-m)}* g(q)\]
in $\Lambda\setminus(x+y\s)$. Moreover, if $n>0$ then $g(p)\neq0$.
\end{theorem}

\begin{proof}
Let $\mathcal{S}$ denote the set of the nonremovable poles of $f$. If $C\subseteq\Omega\setminus\mathcal{S}$, then $m=0=n$ and the thesis immediately follows. Let us, therefore, suppose $C$ to intersect $\mathcal{S}$. We assume, without loss of generality, $\mathcal{S}\subseteq C$.

We can define a slice regular function $h:\Omega\setminus C\to\hh$ by setting $h(q):=\left[(q-x)^2+y^2\right]^nf(q)$. Every $\widetilde p=x+yJ\in C$ is a singularity for $h$ with $ord_f(\widetilde p)\leq n$. Since $h_J(z)=\big(z-\widetilde p\big)^n\big(z-\overline{\widetilde p}\big)^nf_J(z)$ for all $z\in\Omega_J\setminus\{\widetilde p\}$, it follows that $h_J$ extends to a holomorphic function on $\Omega_J$. We conclude that $h$ extends to a slice regular function on $\Omega$.

Finally, since $ord_f(p)=m$ and $h_I(z)=(z-p)^n(z-\bar p)^nf_I(z)$ for all $z\in\Omega_I$, it follows that $(z-p)^{n-m}$ divides $h_I(z)$. As a consequence, $(q-p)^{*(n-m)}$ divides $h(q)$ near $C$ and the thesis follows.
\end{proof}

\begin{proposition}\label{prop:algebraofsemiregular}
The set of semiregular functions on a slice domain $\Omega$ is a real $*$-algebra with respect to $+,*,^c$. Moreover, it is a division ring.
\end{proposition}

\begin{proof}
Let $f,g$ be two semiregular functions on $\Omega$ and let $\mathcal{S}_f, \mathcal{S}_g$ denote the sets of nonremovable singularities of $f,g$ (respectively). 

The function $f$ and its regular conjugate $f^c$ are slice regular in $\Omega\setminus\mathcal{S}_f$. If, moreover, $f \not \equiv 0$, then the regular reciprocal $f^{-*}$ is defined (and slice regular) on $\Omega \setminus (\mathcal{S}_f \cup Z(f^s))$. Now, $f^c$ and $f^{-*}$ (if the latter is defined) are semiregular in $\Omega$ by the following reasoning.
\begin{itemize}
\item For any $p=x+yI\in\mathcal{S}_f$, we can apply Theorem~\ref{polefactorization1} to find natural numbers $m,n$ with $m\leq n$, a slice domain $\Lambda$ with $p\in\Lambda\subseteq\Omega$ and a slice regular function $g : \Lambda \to \hh$ with $g(p)\neq0$ such that
\[f(q)=\left[(q-x)^2+y^2\right]^{-n}(q-p)^{*(n-m)}*g(q)\]
in $\Lambda\setminus(x+y\s)$. The equality
\[f^c(q)=\left[(q-x)^2+y^2\right]^{-n}*g^c(q)*(q-\bar p)^{*(n-m)}\,,\]
valid in $\Lambda\setminus(x+y\s)$, implies that $ord_{f^c}(p)\leq n$. Since $m\leq n$, the chain of equalities
\begin{align*}
f^{-*}(q)&=\left[(q-x)^2+y^2\right]^n*g^{-*}(q)*(q-p)^{*(m-n)}\\
&=\left[(q-x)^2+y^2\right]^m*g^{-*}(q)*(q-\bar p)^{*(n-m)}\,,
\end{align*}
valid in $\Lambda\setminus(x+y\s)$, implies that $ord_{f^{-*}}(p)=0$.
\item For $p_0=x_0+y_0I_0\in Z(f^s)$, let $C$ denote the connected component of $(x_0+y_0\s)\cap\Omega$ that includes $p_0$. By Theorem~\ref{factorization}, there exist $k\in\nn$ and a slice preserving regular function $h$ (which does not vanish in $C$) such that
\[f^s(q)=\left[(q-x_0)^2+y_0^2\right]^kh(q)\]
near $C$. The equality
\begin{align*}
f^{-*}(q)&=(f^s)^{-*}*f^c(q)\\
&=\left[(q-x_0)^2+y_0^2\right]^{-k}h^{-*}*f^c(q)\,,
\end{align*}
valid in a deleted neighborhood of $C$, implies that $ord_{f^{-*}}(p_0)\leq k$.
\end{itemize}

The sum $f+g$ and the regular product $f*g$ are defined (and slice regular) on the largest slice domain in which both $f$ and $g$ are slice regular, namely $\Omega \setminus (\mathcal{S}_f \cup \mathcal{S}_g)$. Reasoning as before, we find that $f+g$ and $f*g$ are semiregular in $\Omega$.

The set of semiregular functions on $\Omega$ is a real $*$-algebra and a division ring as a consequence of the analogous properties of the set of slice regular functions on an arbitrary slice domain $\Lambda$.
\end{proof}

\begin{theorem}\label{R-transfactorization}
Let $\Omega$ be a slice domain in $\hh$, let $f$ be semiregular in $\Omega$ and suppose $f \not \equiv 0$. Fix an $S:=x+y\s$ intersecting $\Omega$ and a connected component $C$ of $S\cap\Omega$. There exist $m \in \zz, n \in \nn$, $p_1,...,p_n \in S$ (with $p_i \neq \bar p_{i+1}$ for all $i \in \{1,\ldots, n-1\}$) such that
\begin{equation}
f(q) = [(q-x)^2+y^2]^m (q-p_1)*(q-p_2)*...*(q-p_n)*g(q)
\end{equation}
in $\Omega$, for some semiregular function $g$ on $\Omega$ that is slice regular near $C$ and has $g^s_{|_C}\not\equiv0$.
\end{theorem}

\begin{proof}
Let $M:=\max_{\widetilde p\in C}ord_f(\widetilde p)$. The previous proposition allows us to define a semiregular $h$ on $\Omega$ by setting
\[h(q):=[(q-x)^2+y^2]^Mf(q)\,.\]
Moreover, $h$ is slice regular near $C$ by Theorem~\ref{polefactorization1}. We can apply Theorem~\ref{factorization} to find $N,n\in\nn$, $p_1,...,p_n \in S$ (with $p_i \neq \bar p_{i+1}$ for all $i \in \{1,\ldots, n-1\}$), and a slice regular function $g$ (with $g^s_{|_C}\not\equiv0$) such that
\[h(q)=[(q-x)^2+y^2]^N (q-p_1)*(q-p_2)*...*(q-p_n)*g(g)\]
near $C$. We point out that $N$ can be greater than zero only if $M=0$. Now, we can extend $g$ to a semiregular function on $\Omega$ by setting
\[g:=R^{-*}*h,\quad R(q):=[(q-x)^2+y^2]^N (q-p_1)*(q-p_2)*...*(q-p_n)\,.\]
Setting $m:=N-M$ concludes the proof.
\end{proof}

\begin{definition}\label{R-spherical order}
In the situation presented in Theorem \ref{R-transfactorization}: if $m\leq0$ then we call $-2m$ the \emph{spherical order of $f$ at $C$}, and write $ord_f^C(x+y\s)=-2m$; whenever $n>0$, we say that $n$ is the \emph{isolated multiplicity of $f$ at $p_1$, relative to $C$}.  If, moreover, $S\cap\Omega$ has a unique connected component, then we also call $-2m$ the \emph{spherical order of $f$ at $S\cap\Omega$} and denote it as $ord_f(x+y\s)$; we call $n$ the \emph{isolated multiplicity of $f$ at $p_1$}.
\end{definition}


\section{Minimum Modulus Principle and Open Mapping Theorem}\label{sec:min}

This section describes some topological properties of slice regular functions on slice domains. We begin by recalling the Maximum Modulus Principle,~\cite[Theorem 7.1]{librospringer}.

\begin{theorem}[Maximum Modulus Principle]\label{maximum}
Let $\Omega$ be a slice domain in $\hh$ and let $f:\Omega\to\hh$ be a slice regular function. If $|f|$ has a relative maximum at $p \in \Omega$, then $f$ is constant.
\end{theorem}

We now prove the Minimum Modulus Principle over slice domains, exploiting the analogous result for symmetric slice domains, namely~\cite[Theorem 7.3]{librospringer}.

\begin{theorem}[Minimum Modulus Principle]\label{minimum}
Let $\Omega$ be a slice domain in $\hh$ and let $f:\Omega\to\hh$ be a slice regular function. If $|f|$ has a relative minimum at $p \in \Omega$, then either $f(p)=0$ or $f$ is constant.
\end{theorem}

\begin{proof}
We can apply Theorem~\ref{localextension} to find an extension triplet $\big(\widetilde f,N,\Lambda\big)$ for $f$, with $p\in\Lambda$. Clearly, $\big|\widetilde f\big|$ has a relative minimum at $p$. By~\cite[Theorem 7.3]{librospringer}, either $\widetilde f(p)=0$ or $\widetilde f$ is constant. In the former case, $f(p)=0$. In the latter case, $f$ is constant in $\Lambda$, whence in $\Omega$ by the Identity Principle~\ref{identity}.
\end{proof}

We now come to the Open Mapping Theorem. For the case of symmetric slice domains, see~\cite[\S7.2]{librospringer}. We begin with some preliminary material.

\begin{definition}
Let $\Omega$ be a slice domain in $\hh$ and let $f:\Omega\to\hh$ be a slice regular function. The zero set of the spherical derivative $f'_s:\Omega\setminus\rr\to\hh$ is called the \emph{degenerate set} of $f$ and denoted by the symbol $D_f$.
\end{definition}

\begin{proposition}
Let $\Omega$ be a slice domain in $\hh$ and let $f:\Omega\to\hh$ be a slice regular function. Either $f$ is constant or $D_f$ is a proper analytic subset of $\Omega\setminus\rr$. In the latter case, the dimension of $D_f$ cannot exceed $3$.
\end{proposition}

\begin{proof}
Since $f'_s$ is a real analytic function, $D_f$ is an analytic subset of $\Omega\setminus\rr$. If this subset is proper, then its dimension is less than, or equal to, $3$. Assume, instead, $D_f=\Omega\setminus\rr$. Consider an extension triplet $\big(\widetilde f,N,\Lambda\big)$ for $f$. Clearly, $D_{\widetilde f}$ includes the nonempty open set $\Lambda\setminus\rr$. If we apply~\cite[Theorem 6.4]{gporientation} (or~\cite[Proposition 4.12]{altavillawithoutreal}) to the function $\widetilde f$ on the symmetric slice domain $N$, we can conclude that $\widetilde f$ is constant. Thus, $f$ is constant in $\Lambda$, whence in $\Omega$ by the Identity Principle~\ref{identity}.
\end{proof}

We point out that, when $\Omega$ is not symmetric, the degenerate set $D_f$ needs not be symmetric. Indeed, the spherical derivative $f'_s$ may vanish identically on a connected component of $(x+y\s)\cap\Omega$ while not vanishing elsewhere in $(x+y\s)\cap\Omega$. This happens, for instance, in Example~\ref{ex:1cap1point}.

We are now ready for the announced theorem. In the statement, $\overline{D_f}$ denotes the closure in $\Omega$ of $D_f\subseteq\Omega\setminus\rr$. We point out that the added set $\overline{D_f}\setminus D_f$ is included in the intersection between the singular set of $f$ and $\rr$, whence in the discrete set $Z(f'_c)\cap\rr$.

\begin{theorem}[Open Mapping]\label{open}
Let $\Omega$ be a slice domain in $\hh$, let $f:\Omega\to\hh$ be a nonconstant slice regular function and let $D_f$ be its degenerate set. Then $f:\Omega \setminus \overline{D_f} \to \hh$ is open.
\end{theorem}

\begin{proof}
Let $U$ be an open subset of $\Omega\setminus\overline{D_f}$, pick $q_0\in U$ and let us prove that $p_0=f(q_0)$ is an interior point of $f(U)$. To do so, we will find $\varepsilon>0$ such that $f(U)$ includes the open Euclidean ball $B(p_0,\varepsilon)$ of radius $\varepsilon$ centered at $p_0$. By construction, either $q_0\in\rr$ or $f'_s(q_0)\neq0$. In either case, by Corollary~\ref{cor:zeros}, the point $q_0$ is an isolated zero of the function $f(q)-p_0$. Thus, there exists $r>0$ such that $\overline{B(q_0,r)} \subseteq U$ and $f(q)-p_0\neq0$ for all $q\in\overline{B(q_0,r)}\setminus\{q_0\}$. This implies that
\[\min_{|q-q_0| = r} |f(q)-p_0|\geq3\varepsilon\]
for some $\varepsilon>0$. If $p\in B(p_0,\varepsilon)$ and if $|q-q_0|=r$, it holds 
\[|f(q) -p| \geq |f(q) -p_0| - |p-p_0| \geq 3 \varepsilon -\varepsilon = 2 \varepsilon > \varepsilon \geq |p_0-p| = |f(q_0)-p|.\]
It follows that
\[\min_{|q-q_0|=r}|f(q) -p|>|f(q_0)-p|,\] 
whence $|f(q)-p|$ has a local minimum point $q_1\in B(q_0,r)$. Taking into account that $f(q)-p$ is not constant, Theorem \ref{minimum} implies that $f(q_1) = p$. We conclude that $f(U) \supseteq B(p_0,\varepsilon)$, as desired.
\end{proof}


\section{Integral representation}\label{sec:integral}

This section studies integral representation formulas for slice regular functions on slice domains. For the case of symmetric slice domains, see~\cite[\S6.2]{librospringer} (derived from~\cite{cauchy,advances}) and~\cite[Corollary 2.6]{volumeintegral}. Let us begin with some notations. 

Let ${\gamma_I} : [0,1] \to L_I$ be a rectifiable curve whose support lies in a plane $L_I$ for some $I \in \s$, let $\Gamma_I$ be a neighborhood of ${\gamma_I}$ in $L_I$ and let $f,g : \Gamma_I \to \hh$ be continuous functions. If $J \in \s$ is such that $J \perp I$ then $L_I + L_IJ = \hh =L_I + J L_I$ and there exist continuous functions $F,G,H,K:\Gamma_I \to L_I$ such that $f = F+GJ$ and $g = H+JK$ in $\Gamma_I$. We use the notation
\begin{align*}
\int_{\gamma_I} g(s) ds f(s) :=& \int_{\gamma_I} H(s) ds F(s) +  \int_{\gamma_I} H(s) dsG(s) J+\\
& +  J \int_{\gamma_I} K(s)ds F(s)+ J \int_{\gamma_I} K(s)ds G(s)J\,.
\end{align*}

We now consider the ``slicewise'' Cauchy Integral Formula.

\begin{lemma}\label{slicewisecauchy}
Let $f:\Omega\to\hh$ be a slice regular function, let $I \in \s$ and let $U_I$ be a bounded Jordan domain in $L_I$, with $\overline{U_I} \subset \Omega_I$. If $\partial U_I$ is rectifiable then 
\[f(z) = \frac{1}{2 \pi I} \int_{\partial U_I} \frac{ds}{s-z} f(s)\]
for all $z \in U_I$.
\end{lemma}

For the previous statement, the proof of~\cite[Lemma 6.3]{librospringer} holds verbatim without assuming $\Omega$ to be a symmetric slice domain (an assumption made in the original statement).

We now derive a local version of the Cauchy Formula, valid on all slice domains. Another version will be provided in Theorem~\ref{cauchyformula2}. For each $s\in\hh$, let $(s-q)^{-*}$ denote the regular reciprocal of $q \mapsto s-q$, i.e.
\[(s-q)^{-*} = (|s|^2-q2Re(s)+ q^2)^{-1}(\bar s-q)\,.\]

\begin{proposition}[Local Cauchy Formula I]\label{cauchyformula}
Let $\Omega$ be a slice domain in $\hh$, let $f:\Omega\to\hh$ be a slice regular function and let $\big(\widetilde f,N,\Lambda\big)$ be an extension triplet for $f$. If $U$ is a bounded symmetric open subset of $\hh$ with $\overline{U}\subseteq N$, if $I\in\s$ and if the boundary $\partial U_I$ is a finite union of disjoint rectifiable Jordan curves, then
\[f(q) =\int_{\partial U_I} (s-q)^{-*} (2\pi I)^{-1}ds\,\widetilde f(s)\]
for all $q\in U\cap\Lambda$.
\end{proposition}

\begin{proof}
By~\cite[Theorem 6.4]{librospringer}, the formula
\[\widetilde f(q) =\int_{\partial U_I} (s-q)^{-*} (2\pi I)^{-1}ds\,\widetilde f(s)\]
applies to all $q\in U$. Since $\widetilde f$ coincides with $f$ in $\Lambda$, the thesis follows.
\end{proof}

The same technique can be applied to prove the next result. For any $x,y\in\rr$ and any $I\in\s$, we will use the kernel
\[C(q,x+yI):=(2\pi y)^{-2}\left[(q-x)^2+y^2\right]^{-1}(x-yI-q)=(2\pi y)^{-2}(x+Iy-q)^{-*}\]
for $q\in\hh$.

\begin{proposition}[Local Volume Cauchy Formula]
Let $\Omega$ be a slice domain in $\hh$, let $f:\Omega\to\hh$ be a slice regular function and let $\big(\widetilde f,N,\Lambda\big)$ be an extension triplet for $f$. Let $U$ be a bounded symmetric open subset of $\hh$ with $\overline{U}\subseteq N$ and assume the boundary $\partial U$ to be $C^1$. For $w\in\partial U$, let ${\bf n}(w)$ denote the outer normal versor of $\partial U$ at $w$ and let $d\sigma_w$ denote the standard $3$-volume form on $\partial U$. Then
\[f(q)=\int_{\partial U} C(q,w){\bf n}(w)\widetilde f(w)d\sigma_w\]
for all $q\in U\cap\Lambda$.
\end{proposition}

\begin{proof}
By~\cite[Corollary 2.6]{volumeintegral}, the formula
\[\widetilde f(q)=\int_{\partial U} C(q,w){\bf n}(w)\widetilde f(w)d\sigma_w\]
applies to all $q\in U$. Since $\widetilde f$ coincides with $f$ in $\Lambda$, the thesis follows.
\end{proof}

We now come to the announced second local version of the Cauchy Formula. In this second version, the set $U$ does not depend on the choice of an extension triplet but only on the domain $\Omega$. Moreover, a subset of $U$ where the formula holds can be described explicitly. To prove this second version, we need to strengthen Theorem~\ref{localextension} to the next result, which is of independent interest.

\begin{theorem}[Local Extension]\label{localextension2}
Let $f$ be a slice regular function on a slice domain $\Omega$ and let $J_0\in\s$. Let $\mathcal{C}$ be a symmetric, compact and path-connected subset of $\hh$ such that
\[\emptyset\neq\mathcal{C}_{J_0}^+\subset\Omega_{J_0}^+\,.\]
If $\mathcal{C}\cap\rr$ is not empty, suppose it is a closed interval included in $\Omega$. Then there exist an extension triplet $\big(\widetilde f,N,\Lambda\big)$ for $f$ and a real number $\delta>0$ such that
\[\bigcup_{|J-J_0|\leq\delta} \overline{\mathcal{C}_J^+}\subset\Lambda\,.\]
In particular, $\widetilde f$ coincides with $f$ in a neighborhood of $\overline{\mathcal{C}_{J_0}^+}$.
\end{theorem}

\begin{proof}
We assume, without loss of generality, $\mathcal{C}$ to intersect $\rr$ in a closed interval included in $\Omega$. Indeed, if $\mathcal{C}\cap\rr=\emptyset$, we may proceed as follows. Since $\Omega$ is a slice domain, we can pick a path $\gamma:[0,1]\to\Omega_{J_0}$ with $\gamma(0)\in\rr, \gamma((0,1])\subset\Omega_{J_0}^+$ and $\gamma(1)\in\mathcal{C}_{J_0}^+$. We can denote by $P$ the symmetric completion of the support of $\gamma$ and replace $\mathcal{C}$ by $P\cup\mathcal{C}$. In other words, we can replace every half slice $\mathcal{C}_J^+$ by the ``kite'' $P_J^+\cup\mathcal{C}_J^+$.

We set $C:=\overline{\mathcal{C}_{J_0}^+}=(\mathcal{C}\cap\rr)\cup\mathcal{C}_{J_0}^+$ and remark that $C$ is a compact and path-connected subset of $\Omega_{J_0}$. By Lemma~\ref{gamma}, there exists $\varepsilon>0$ such $M := \Gamma(C,\varepsilon)$ is a slice domain with the property $C \subset M \subseteq \Omega$. Let $q_0\in C$ be such that $\max_{q\in C}|\im(q)|=|\im(q_0)|$.

Let $x+J_0y\in C$. If we choose $K_0 \in \s$ with $0<|K_0-J_0|<\frac{\varepsilon}{|\im(q_0)|}$, then $x+K_0y$ is included in $M_{K_0}$. Indeed, the distance between $x+K_0y$ and $x+J_0y$ is $y|K_0-J_0|$, which is less than $\frac{y}{|\im(q_0)|}\varepsilon$ if $y>0$ and is $0$ if $y=0$.

Let $N$ be the symmetric completion of the connected set $M_{K_0}$. We point out the following properties of $N$: it includes $\mathcal{C}$; it has $N_{K_0}^+=M_{K_0}^+$; and it is a slice domain. Moreover, $N_{J_0}^+ \subseteq M_{J_0}^+ \subseteq \Omega_{J_0}$. Indeed, for each $x+yJ_0 \in N_{J_0}^+$ it holds $x+yK_0 \in N_{K_0}^+=M_{K_0}^+$. By direct computation, for all $q\in C$, it holds
\[|x+yK_0-q|^2-|x+yJ_0-q|^2 = 2y|\im(q)|(1-\langle K_0, J_0 \rangle) \geq 0 \,.\]
Thus, the distance between $x+yJ_0$ and $q$ is less than, or equal to, the distance between $x+yK_0$ and $q$. If $q\in C\cap\rr$ and if $B(q,\varepsilon)$ includes $x+yK_0$, then the same ball includes $x+yJ_0$. Similarly, if $q\in C\setminus\rr$ and if $B\left(q,\frac{|\im(q)|}{|\im(q_0)|}\varepsilon\right)$ includes $x+yK_0$, then the same ball includes $x+yJ_0$. In both cases, $x+yJ_0$ belongs to $M_{J_0}^+$.

By the General Extension Formula~\ref{extensionformulathm}, there exists a unique slice regular function $\widetilde f: N \to \hh$ that coincides with $f_{J_0}$ in $N_{J_0}^+ \subseteq M_{J_0}^+ \subseteq \Omega_{J_0}$, with $f_{K_0}$ in $N_{K_0}^+=M_{K_0}^+$ and with $f$ in $N \cap \rr = M_{K_0} \cap \rr$.

Within the open set $N \cap \Omega$, the slice $(N \cap \Omega)_{J_0}$ includes $C$. Lemma~\ref{gamma} guarantees that there exists $\varepsilon_0>0$ such that the slice domain $\Lambda := \Gamma(C,\varepsilon_0)$ has the property $C \subset \Lambda \subseteq N\cap\Omega$. Now, $\widetilde f$ and $f$ coincide in $N \cap \Omega \cap \rr = N \cap \rr$, whence throughout $\Lambda$ by the Identity Principle~\ref{identity}.

Finally, let us prove that for all $J\in\s$ with $|J-J_0|<\frac{\varepsilon_0}{|\im(q_0)|}$, it holds $\overline{\mathcal{C}_J^+}\subset\Lambda$: if $x\in\mathcal{C}\cap\rr$, then $x\in C\cap\rr\subset\Lambda$; if the point $x+Jy$ belongs to $\mathcal{C}_J^+$, then its distance from $x+J_0y$ is $y|J-J_0|<\frac{y}{|\im(q_0)|}\varepsilon_0$, whence $x+Jy\in B\left(x+J_0y,\frac{y}{|\im(q_0)|}\varepsilon_0\right)\subseteq\Lambda$, as desired.
\end{proof}

We are now ready for the announced second local version of the Cauchy Formula. 

\begin{theorem}[Local Cauchy Formula II]\label{cauchyformula2}
Let $\Omega$ be a slice domain in $\hh$ and let $f:\Omega\to\hh$ be a slice regular function. Let $U$ be a symmetric open subset of $\hh$ such that $\overline{U}$ is compact and path-connected. If $\overline{U}\cap\rr\neq\emptyset$, we assume the same intersection to be a closed interval of $\rr$ included in $\Omega$. Suppose, for some $J_0\in\s$, that the boundary $\partial U_{J_0}$ is a finite union of disjoint rectifiable Jordan curves and that
\[\overline{U}_{J_0}^+\subset\Omega_{J_0}^+\,.\]
Then, there exists a real number $\varepsilon>0$ such that, for all $I\in\s$ and for all
\[q\in\bigcup_{|J-J_0|<\varepsilon} U_J^+\cup(U\cap\rr)\,,\]
it holds
\[f(q) =\int_{\partial U_I} (s-q)^{-*} (2\pi I)^{-1}ds\,\widetilde f(s)\,,\]
where, for all $x+yJ\in\partial U$ (with $x,y\in\rr,y>0$ and $J\in\s$), it holds
\[\widetilde f(x+yJ)=f^\circ_s(x+yJ_0)+yJf'_s(x+yJ_0)\]
and for all $x\in\partial U\cap\rr$ it holds $\widetilde f(x)=f^\circ_s(x)=f(x)$.
\end{theorem}

\begin{proof}
By Theorem~\ref{localextension2}, there exist an extension triplet $\big(\widetilde f,N,\Lambda\big)$ for $f$ and a real number $\delta>0$ such that
\[\bigcup_{|J-J_0|\leq\delta} \overline{U_J^+}\subset\Lambda\,.\]
Since $N$ is a symmetric slice domain including $\Lambda$, it follows that $N\supset\overline{U}$. We can apply~\cite[Theorem 6.4]{librospringer} and find that 
\[\widetilde f(q) =\int_{\partial U_I} (s-q)^{-*} (2\pi I)^{-1}ds\,\widetilde f(s)\]
for all $q\in U$ and for all $I\in\s$. Now, $\widetilde f(q)=f(q)$ for all $q$ in
\[U\cap\Lambda\supset\bigcup_{|J-J_0|\leq\delta} U_J^+\cup(U\cap\rr)\,.\]
Moreover, for all $x\in\Lambda\cap\rr$ it holds $\widetilde f(x)=f(x)$. Finally, for all $x+yJ_0\in\Lambda$ with $y>0$ and for all $J\in\s$ it holds
\begin{align*}
\widetilde f(x+yJ)&=\widetilde f^\circ_s(x+yJ)+yJ\widetilde f'_s(x+yJ)\\
&=\widetilde f^\circ_s(x+yJ_0)+yJ\widetilde f'_s(x+yJ_0)\\
&=f^\circ_s(x+yJ_0)+yJf'_s(x+yJ_0)\,.\qedhere
\end{align*}
\end{proof}


\section{Spherical series expansions}\label{sec:sphericalseries}

This section provides \emph{spherical expansions} for slice regular functions over slice domains. For the case of symmetric slice domains, see~\cite[\S8.1]{librospringer} (which derived from~\cite{expansion}) and~\cite[Theorem 6.2]{gpsdivisionalgebras} (which derived from~\cite{gpssingularities}). We will use the open \emph{Cassini ball} of radius $R>0$ centered at $x_0+y_0\s$, namely
\[U(x_0+y_0\s,R) = \{q \in \hh : |(q-x_0)^2+y_0^2| < R^2\}\,.\]

\begin{theorem}
Let $\Omega$ be a slice domain in $\hh$ and let $f:\Omega\to\hh$ be a slice regular function. If $x_0+y_0\s$ intersects $\Omega$ and if $C$ is a connected component of the intersection, then there exist an open neighborhood $U$ of $C$ and quaternions $\{a_k\}_{k\in\nn}$ such that
\[f(q) = \sum_{n \in \nn}[(q-x_0)^2+y_0^2]^n [a_{2n} + qa_{2n+1}]\]
for all $q\in U$.
\end{theorem}

\begin{proof}
For any $q_0\in C$, by Theorem~\ref{localextension}, there exists an extension triplet $\tau_0=\big(\widetilde f_0,N_0,\Lambda_0\big)$ for $f$ such that $q_0\in\Lambda_0$. We can apply~\cite[Corollary 8.10]{librospringer} to $\widetilde f_0$ and find $R >0$, as well as quaternions $\{a^{\tau_0}_k\}_{k\in\nn}$, such that
\[\widetilde f_0(q) = \sum_{n \in \nn}[(q-x_0)^2+y_0^2]^n [a^{\tau_0}_{2n} + qa^{\tau_0}_{2n+1}]\]
in the open neighborhood $U(x_0+y_0\s,R)$ of $x_0+y_0\s$. It follows that
\[f(q) = \sum_{n \in \nn}[(q-x_0)^2+y_0^2]^n [a^{\tau_0}_{2n} + qa^{\tau_0}_{2n+1}]\]
in the neighborhood $\Lambda_0\cap U(x_0+y_0\s,R)$ of $q_0$.

Let us now assume that $q_0,q_1\in C$, that $\tau_0=\big(\widetilde f_0,N_0,\Lambda_0\big),\tau_1=\big(\widetilde f_1,N_1,\Lambda_1\big)$ are extension triplets for $f$ with $q_0\in\Lambda_0,q_1\in\Lambda_1$, and that $\Lambda_0\cap\Lambda_1\cap C\neq\emptyset$. It can be checked by induction that the equality 
\[\sum_{n \in \nn}[(q-x_0)^2+y_0^2]^n [a^{\tau_0}_{2n} + qa^{\tau_0}_{2n+1}]=\sum_{n \in \nn}[(q-x_0)^2+y_0^2]^n [a^{\tau_1}_{2n} + qa^{\tau_1}_{2n+1}]\,,\]
valid in $\Lambda_0\cap\Lambda_1\cap U(x_0+y_0\s,r)$ for some $r>0$, implies that $a^{\tau_0}_k=a^{\tau_1}_k$ for all $k\in\nn$.

Finally, for every $q_0,\widetilde q_0\in C$ it is possible to find a finite number of extension triplets $\tau_0,\ldots,\tau_m$ for $f$, with $\tau_t=\big(\widetilde f_t,N_t,\Lambda_t\big)$, with $q_0\in\Lambda_0,\widetilde q_0\in\Lambda_m$ and with $\Lambda_{t-1}\cap\Lambda_t\cap C\neq\emptyset$ for all $t\in\{1,\ldots,m\}$. It follows that the coefficients $\{a^{\tau_0}_k\}_{k\in\nn}$ of the expansion valid near $q_0$ coincide with the coefficients $\{a^{\tau_m}_k\}_{k\in\nn}$ of the expansion valid near $\widetilde q_0$. The thesis follows.
\end{proof}

We now come to spherical Laurent expansions. For $0\leq r_1<r_2$, we will use the open \emph{Cassini shell} of inner radius $r_1$ and outer radius $r_2$ centered at $x_0+y_0\s$, namely
\[U(x_0+y_0\s,r_1,r_2):=\{q\in\hh : r_1^2<|(q-x_0)^2+y_0^2|<r_2^2\}\,.\]
We begin by focusing on the case $r_1>0$.

\begin{theorem}\label{thm:sphericallaurent}
Let $\Omega$ be a slice domain in $\hh$ and let $f:\Omega\to\hh$ be a slice regular function. Suppose $x_0,y_0,r_1,r_2\in\rr$ to fulfill the inequalities $0<r_1<r_2\leq y_0$ and set $U:=U(x_0+y_0\s,r_1,r_2)$. If there exists $J_0\in\s$ such that
\[\overline{U_{J_0}^+}\subset\Omega\,,\]
then there exist a real number $\varepsilon>0$ and quaternions $\{a_k\}_{k\in\zz}$ such that
\[f(q) = \sum_{n \in \zz}[(q-x_0)^2+y_0^2]^n [a_{2n} + qa_{2n+1}]\]
for all $q$ belonging to the open neighborhood
\[\bigcup_{|J-J_0|<\varepsilon} U_J^+\]
of $U_{J_0}^+$.
\end{theorem}

\begin{proof}
The closure of $U$, namely,
\[\overline{U}=\{q\in\hh : r_1^2\leq|(q-x_0)^2+y_0^2|\leq r_2^2\}\,\]
is a symmetric, compact and path-connected subset of $\hh$. If $r_2=y_0$, then $\overline{U}\cap\rr=\{x_0\}$. If, instead, $r_2<y_0$, then $\overline{U}\cap\rr=\emptyset$. 

Let us apply Theorem~\ref{localextension2} with $\mathcal{C}:=\overline{U}$. There exist an extension triplet $\big(\widetilde f,N,\Lambda\big)$ for $f$ and a real number $\delta>0$ such that $\Lambda$ includes the union
\[\bigcup_{|J-J_0|\leq\delta} \overline{\mathcal{C}_J^+}\]
of the slices $\overline{\mathcal{C}_J^+}=(\overline{U}\cap\rr)\cup\overline{U}_J^+$ of $\mathcal{C}$.

Since $N$ is a symmetric slice domain including $\Lambda$, it follows that $N\supset\mathcal{C}=\overline{U}$. By~\cite[Theorem 6.2]{gpsdivisionalgebras}, there exist quaternions $\{a_k\}_{k\in\zz}$ such that
\[\widetilde f(q) = \sum_{n \in \zz}[(q-x_0)^2+y_0^2]^n [a_{2n} + qa_{2n+1}]\]
for all $q\in U$.

Since $\widetilde f$ coincides with $f$ in 
\[\Lambda\cap U\supseteq\bigcup_{|J-J_0|\leq\delta} U_J^+\,,\]
the thesis follows.
\end{proof}

We can complete Theorem~\ref{thm:sphericallaurent} in two different senses. The last sentence in~\cite[Theorem 6.2]{gpsdivisionalgebras} allows the next remark for the case of a Cassini ball.

\begin{remark}
The statement of Theorem~\ref{thm:sphericallaurent} holds true if we substitute $U(x_0+y_0\s,r_2)$ for $U(x_0+y_0\s,r_1,r_2)$. In such a case, $a_k=0$ for all $k<0$.
\end{remark}

For the case $r_1=0$, i.e., the case of a Cassini shell obtained from a Cassini ball by erasing a single $2$-sphere, we can prove the next result.

\begin{theorem}\label{thm:sphericallaurent2}
Let $\Omega$ be a slice domain in $\hh$ and let $f:\Omega\to\hh$ be a slice regular function. Suppose $x_0,y_0,r_2\in\rr$ to fulfill the inequalities $0<r_2\leq y_0$ and set
\[V:=U(x_0+y_0\s,0,r_2)=U(x_0+y_0\s,r_2)\setminus(x_0+y_0\s)\,.\]
If there exist $I_0\in\s$ and $\delta>0$ such that
\[\bigcup_{|J-I_0|<\delta} \overline{V_J^+}\setminus\{x_0+y_0J\}\subset\Omega\,,\]
then there exist quaternions $\{a_k\}_{k\in\zz}$ such that
\[f(q) = \sum_{n \in \zz}[(q-x_0)^2+y_0^2]^n [a_{2n} + qa_{2n+1}]\]
for all $q$ belonging to
\[\bigcup_{|J-I_0|<\delta} V_J^+\,.\]
\end{theorem}

\begin{proof}
Fix $r_1$ with $0<r_1<r_2$. Let us apply Theorem~\ref{thm:sphericallaurent} with $U:=U(x_0+y_0\s,r_1,r_2)$ at each $J_0\in\s$ with $|J_0-I_0|<\delta$: there exist a real number $\varepsilon_{J_0}>0$ and quaternions $\left\{a^{J_0}_k\right\}_{k\in\zz}$ such that
\[f(q) = \sum_{n \in \zz}[(q-x_0)^2+y_0^2]^n [a^{J_0}_{2n} + qa^{J_0}_{2n+1}]\]
for all $q$ belonging to
\[W^{r_1}_{J_0,\varepsilon_{J_0}}:=\bigcup_{|J-J_0|<\varepsilon_{J_0}} U_J^+\,.\]

Let us assume for the moment $\varepsilon_{J_0}\leq2$ for all $J_0\in\s$ such that  $|J_0-I_0|<\delta$. If $J_1\in\s$, too, fulfills the inequality $|J_1-I_0|<\delta$, then it holds $W^{r_1}_{J_0,\varepsilon_{J_0}}\cap W^{r_1}_{J_1,\varepsilon_{J_1}}\neq\emptyset$ if, and only if, if the angle between $J_0$ and $J_1$ is less than $\arccos(1-\varepsilon_{J_0}^2/2)+\arccos(1-\varepsilon_{J_1}^2/2)$ (the latter being the sum of the ``openings'' of $W^{r_1}_{J_0,\varepsilon_{J_0}}$ and $W^{r_1}_{J_1,\varepsilon_{J_1}}$). If this is the case, then the equality
\[\sum_{n \in \zz}[(q-x_0)^2+y_0^2]^n [a^{J_0}_{2n} + qa^{J_0}_{2n+1}]=\sum_{n \in \zz}[(q-x_0)^2+y_0^2]^n [a^{J_1}_{2n} + qa^{J_1}_{2n+1}]\,,\]
valid for all $q\in W^{r_1}_{J_0,\varepsilon_{J_0}}\cap W^{r_1}_{J_1,\varepsilon_{J_1}}$, implies that the sequence $\left\{a^{J_0}_k\right\}_{k\in\zz}$ coincides with the sequence $\left\{a^{J_1}_k\right\}_{k\in\zz}$ by~\cite[Theorem 6.2]{gpsdivisionalgebras}. For each $J_0\in\s$, it is possible to find finitely many $J_1,\ldots,J_m\in\s$ (with $J_m=I_0$) such that $W^{r_1}_{J_t,\varepsilon_{J_t}}\cap W^{r_1}_{J_{t+1},\varepsilon_{J_{t+1}}}\neq\emptyset$ for all $t\in\{0,\ldots,m-1\}$, whence $\left\{a^{J_0}_k\right\}_{k\in\zz}$ coincides with the sequence $\left\{a^{I_0}_k\right\}_{k\in\zz}$. If we write $a_k$ for $a^{I_0}_k$, we have proven that
\[f(q) = \sum_{n \in \zz}[(q-x_0)^2+y_0^2]^n [a_{2n} + qa_{2n+1}]\]
for all $q\in W^{r_1}_{I_0,\delta}$. The case when there exists $J_0\in\s$ such that $|J_0-I_0|<\delta$ and $\varepsilon_{J_0}>2$ is even simpler to treat: in this case, it holds $W^{r_1}_{J_0,\varepsilon_{J_0}}=U\supseteq W^{r_1}_{I_0,\delta}$.

If we repeat this construction for $\widetilde r_1<r_1$ to find an expansion valid in $W^{\widetilde r_1}_{I_0,\delta}$, the coefficients of the new expansion will coincide with $\left\{a_k\right\}_{k\in\zz}$ because $W^{\widetilde r_1}_{I_0,\delta}\supseteq W^{r_1}_{I_0,\delta}$. This proves that
\[f(q) = \sum_{n \in \zz}[(q-x_0)^2+y_0^2]^n [a_{2n} + qa_{2n+1}]\]
for all $q$ belonging to
\[\bigcup_{0<\widetilde r_1<r_1}W^{\widetilde r_1}_{I_0,\delta}=\bigcup_{|J-I_0|<\delta} V_J^+\,.\qedhere\]
\end{proof}

Our results concerning spherical Laurent expansions allow us to make the following remark about singularities.

\begin{remark}\label{rmk:sphericalorder}
Let us assume we are in the situation of Theorem~\ref{thm:sphericallaurent2} and let us denote by $\mathscr{C}$ the spherical cap $\{x_0+y_0J : |J-I_0|<\delta\}$. We can use Theorem~\ref{R-transfactorization} and Definition~\ref{R-spherical order} to observe what follows.

Suppose, for all $n\leq0$, there exists $k<2n$ such that $a_k\neq0$: then every point of $\mathscr{C}$, except at most one, is an essential singularity for $f$.

Suppose, instead, there exists $n\leq0$ such that $a_k=0$ for all $k<2n$ and let $m$ be the maximum such $n$. Notice that, by construction, $\widehat\Omega:=\Omega\cup\mathscr{C}$ is a slice domain. By Proposition~\ref{prop:algebraofsemiregular}, $f$ is semiregular in $\widehat\Omega$. Moreover, if $C$ denotes the connected component of $(x_0+y_0\s)\cap\widehat\Omega$ that includes $\mathscr{C}$ then: $ord_f^C(x_0+y_0\s)=-2m$; $ord_f(p)=-2m$ for all points $p\in C$, except at most one (which must have lesser order). If $m=0$, then every point of $\mathscr{C}$ is a removable singularity for $f$.
\end{remark}

In particular, the hypothesis $\sup_{p\in C}ord_f(p)<+\infty$ taken in Definition~\ref{def:semiregular} is not restrictive if $\Omega\setminus\Omega'$ happens to be a cap within a $2$-sphere of the form $x+y\s$. We can further use Remark~\ref{rmk:sphericalorder} to show that the Casorati-Weierstrass Theorem does not hold for slice regular functions on slice domains that are not symmetric, at least not in the form it has in the complex setting.

\begin{example}
Let $T$ denote the closed line segment from $0$ to $i$ in $\hh$. Consider, as in Lemma~\ref{gamma}, the slice domain
\[\Gamma\left(T,\frac12\right)=B\left(0,\frac12\right)\cup\bigcup_{y\in(0,1]} B\left(yi,\frac{y}{2}\right)\,.\]
For all $I\in\s$, let $\vartheta_I$ denote the angle between $I$ and $i$. If $\vartheta_I\geq\frac\pi6=\arcsin\frac12$, then $\Gamma\left(T,\frac12\right)\cap L_I$ is the disk of radius $\frac12$ centered at $0$ in $L_I$; in particular, $I\not\in\Gamma\left(T,\frac12\right)$. As a consequence, $C:=\Gamma\left(T,\frac12\right)\cap\s$ and its closure in $\s$ are properly contained in $\s$.

Consider the slice domain $\Omega:=\Gamma\left(T,\frac12\right)\setminus C$ and the slice regular function $f: \Omega \to \hh$ defined by the formula 
\[f(q) = \exp((q^2+1)^{-1}) = \sum_{n \in \nn} (q^2+1)^{-n} \frac{1}{n!}\,.\]
Each imaginary unit $I\in C$ is an essential singularity for $f$. Let us consider the open set $U:=\Gamma\left(T,\frac12\right)\setminus\overline{B\left(0,\frac12\right)}$: it includes $C$ and it does not intersect $L_I$ for any $I\in\s$ with $\vartheta_I\geq\frac\pi6$. Since $f$ is slice preserving, for all $I\in\s$ with $\vartheta_I\geq\frac\pi6$ it holds
\[f(U\setminus C)\cap L_I=(f(U\setminus C)\cap\rr)\cup f(U_I\setminus\{I\})\subseteq\rr\cup f(\emptyset)=\rr\,.\]
Thus, $f(U\setminus C)$ cannot be an open dense subset of $\hh$.

Nevertheless, we can observe what follows. The domain $\Omega$ is starlike with respect to $0$, whence a simple domain by Proposition~\ref{prop:starlike}. Consider the slice regular extension $\widetilde f:\widetilde \Omega\to\hh$ of $f$ to the symmetric completion of its domain. The equality $\widetilde f(q):=\exp((q^2+1)^{-1})$ holds for all $q\in\widetilde \Omega$. By~\cite[Theorem 5.33]{librospringer} (derived from~\cite{moebius,singularities}), for every symmetric neighborhood $V$ of $\s$ with $V\setminus\s\subseteq\widetilde \Omega$, the image $\widetilde f(V\setminus\s)$ is dense in $\hh$.
\end{example}





\vfill

\begin{thebibliography}{10}

\bibitem{altavillawithoutreal}
A.~Altavilla.
\newblock Some properties for quaternionic slice regular functions on domains
  without real points.
\newblock {\em Complex Var. Elliptic Equ.}, 60(1):59--77, 2015.

\bibitem{cauchy}
F.~Colombo, G.~Gentili, and I.~Sabadini.
\newblock A {C}auchy kernel for slice regular functions.
\newblock {\em Ann. Global Anal. Geom.}, 37(4):361--378, 2010.

\bibitem{advancesrevised}
F.~Colombo, G.~Gentili, I.~Sabadini, and D.~Struppa.
\newblock Extension results for slice regular functions of a quaternionic
  variable.
\newblock {\em Adv. Math.}, 222(5):1793--1808, 2009.

\bibitem{librodaniele2}
F.~Colombo, I.~Sabadini, and D.~C. Struppa.
\newblock {\em Noncommutative functional calculus. Theory and applications of
  slice hyperholomorphic functions}, volume 289 of {\em Progress in
  Mathematics}.
\newblock Birkh{\"a}user/Springer Basel AG, Basel, 2011.

\bibitem{douren1}
X.~Dou and G.~Ren.
\newblock {R}iemann slice-domains over quaternions {I}.
\newblock arXiv:1808.06994 [math.CV], 2018.

\bibitem{douren2}
X.~Dou and G.~Ren.
\newblock {R}iemann slice-domains over quaternions {II}.
\newblock arXiv:1809.07979 [math.CV], 2018.

\bibitem{dourensabadini}
X.~Dou, G.~Ren, and I.~Sabadini.
\newblock Extension theorem and representation formula in non-axially symmetric
  domains for slice regular functions.
\newblock arXiv:2003.10487 [math.CV], 2020.

\bibitem{ocs}
G.~Gentili, S.~Salamon, and C.~Stoppato.
\newblock Twistor transforms of quaternionic functions and orthogonal complex
  structures.
\newblock {\em J. Eur. Math. Soc. (JEMS)}, 16(11):2323---2353, 2014.

\bibitem{zeros}
G.~Gentili and C.~Stoppato.
\newblock Zeros of regular functions and polynomials of a quaternionic
  variable.
\newblock {\em Michigan Math. J.}, 56(3):655--667, 2008.

\bibitem{open}
G.~Gentili and C.~Stoppato.
\newblock The open mapping theorem for regular quaternionic functions.
\newblock {\em Ann. Sc. Norm. Super. Pisa Cl. Sci. (5)}, VIII(4):805--815,
  2009.

\bibitem{zerosopen}
G.~Gentili and C.~Stoppato.
\newblock The zero sets of slice regular functions and the open mapping
  theorem.
\newblock In I.~Sabadini and F.~Sommen, editors, {\em Hypercomplex analysis and
  applications}, Trends in Mathematics, pages 95--107. Birkh\"auser Verlag,
  Basel, 2011.

\bibitem{powerseries}
G.~Gentili and C.~Stoppato.
\newblock Power series and analyticity over the quaternions.
\newblock {\em Math. Ann.}, 352(1):113--131, 2012.

\bibitem{localrepresentation}
G.~Gentili and C.~Stoppato.
\newblock A local representation formula for quaternionic slice regular
  functions.
\newblock {\em Proc. Amer. Math. Soc.}, to appear.

\bibitem{librospringer}
G.~Gentili, C.~Stoppato, and D.~C. Struppa.
\newblock {\em Regular functions of a quaternionic variable}.
\newblock Springer Monographs in Mathematics. Springer, Heidelberg, 2013.

\bibitem{cras}
G.~Gentili and D.~C. Struppa.
\newblock A new approach to {C}ullen-regular functions of a quaternionic
  variable.
\newblock {\em C. R. Math. Acad. Sci. Paris}, 342(10):741--744, 2006.

\bibitem{advances}
G.~Gentili and D.~C. Struppa.
\newblock A new theory of regular functions of a quaternionic variable.
\newblock {\em Adv. Math.}, 216(1):279--301, 2007.

\bibitem{milan}
G.~Gentili and D.~C. Struppa.
\newblock On the multiplicity of zeroes of polynomials with quaternionic
  coefficients.
\newblock {\em Milan J. Math.}, 76:15--25, 2008.

\bibitem{perotti}
R.~Ghiloni and A.~Perotti.
\newblock Slice regular functions on real alternative algebras.
\newblock {\em Adv. Math.}, 226(2):1662--1691, 2011.

\bibitem{volumeintegral}
R.~Ghiloni and A.~Perotti.
\newblock Volume {C}auchy formulas for slice functions on real associative
  *-algebras.
\newblock {\em Complex Var. Elliptic Equ.}, 58(12):1701--1714, 2013.

\bibitem{global}
R.~Ghiloni and A.~Perotti.
\newblock Global differential equations for slice regular functions.
\newblock {\em Math. Nachr.}, 287(5-6):561--573, 2014.

\bibitem{gporientation}
R.~Ghiloni and A.~Perotti.
\newblock On a class of orientation-preserving maps of $\mathbb{R}^4$.
\newblock {\em J. Geom. Anal.}, doi:10.1007/s12220-020-00356-8, 2020.

\bibitem{gpsalgebra}
R.~Ghiloni, A.~Perotti, and C.~Stoppato.
\newblock The algebra of slice functions.
\newblock {\em Trans. Amer. Math. Soc.}, 369(7):4725--4762, 2017.

\bibitem{gpssingularities}
R.~Ghiloni, A.~Perotti, and C.~Stoppato.
\newblock Singularities of slice regular functions over real alternative
  {$^*$}-algebras.
\newblock {\em Adv. Math.}, 305:1085--1130, 2017.

\bibitem{gpsdivisionalgebras}
R.~Ghiloni, A.~Perotti, and C.~Stoppato.
\newblock Division algebras of slice functions.
\newblock {\em Proc. Roy. Soc. Edinburgh Sect. A}, 150(4):2055--2082, 2020.

\bibitem{conformality}
A.~Gori and F.~Vlacci.
\newblock On a criterion of local invertibility and conformality for slice
  regular quaternionic functions.
\newblock {\em Proc. Edinb. Math. Soc. (2)}, 62(1):97--105, 2019.

\bibitem{perottipersonal}
A.~Perotti.
\newblock Unpublished notes.

\bibitem{poli}
C.~Stoppato.
\newblock Poles of regular quaternionic functions.
\newblock {\em Complex Var. Elliptic Equ.}, 54(11):1001--1018, 2009.

\bibitem{moebius}
C.~Stoppato.
\newblock Regular {M}oebius transformations of the space of quaternions.
\newblock {\em Ann. Global Anal. Geom.}, 39(4):387--401, 2011.

\bibitem{expansion}
C.~Stoppato.
\newblock A new series expansion for slice regular functions.
\newblock {\em Adv. Math.}, 231(3-4):1401--1416, 2012.

\bibitem{singularities}
C.~Stoppato.
\newblock Singularities of slice regular functions.
\newblock {\em Math. Nachr.}, 285(10):1274--1293, 2012.

\end{thebibliography}
\end{document}